\documentclass[12pt,reqno]{amsart}
\usepackage[margin=1in]{geometry}
\usepackage{graphicx}
\usepackage{amsmath,amsthm,amsfonts,mathrsfs,amssymb,float,color}
\usepackage{graphicx}
\usepackage{amsmath,amsfonts,mathrsfs,amssymb,float,color}
\usepackage{graphicx}
\usepackage{textcomp}
\usepackage{verbatim}
\usepackage{upgreek}
\usepackage[T1]{fontenc}
\usepackage[utf8]{inputenc}
\usepackage{hyperref}
\hypersetup{colorlinks=true,citecolor=red,linkcolor=blue}
\usepackage{epsfig,subfigure,fancybox,balance}
\usepackage{diagbox}
\usepackage{youngtab}

\newcommand{\beq}[1]{\begin{equation} \label{#1}}
\newcommand{\eeq}{\end{equation}}
\newcommand{\bea}{\bed\begin{array}{rl}}
\newcommand{\eea}{\end{array}\eed}
\newcommand{\bed}{\begin{displaymath}}
\newcommand{\eed}{\end{displaymath}}
\newcommand{\barray}{\begin{array}{ll}}
\newcommand{\earray}{\end{array}}
\newcommand{\disp}{\displaystyle}
\newcommand{\ad}{&\!\disp}
\newcommand{\aad}{&\disp}
\newcommand{\beqa}[1]{\begin{equation}\label{#1}\barray}
\newcommand{\eeqa}{\earray\end{equation}}
\newcommand{\al}{\alpha}
\newcommand{\e}{\varepsilon}

\newcommand{\la}{\lambda}
\newcommand{\La}{\Lambda}
\newcommand{\sg}{\sigma}

\newcommand{\ga}{\gamma}
\newcommand{\Ga}{\Gamma}
\newcommand{\dl}{\delta}
\newcommand{\Dl}{\Delta}
\newcommand{\cd}{(\cdot)}

\def\phi{\varphi}
\def\indi{{\bf 1}}
\def\half{\frac{1}{2}}

\newcommand{\CA}{{\mathcal A}}
\newcommand{\CF}{{\mathcal F}}

\newcommand{\CU}{\mathcal{U}}
\newcommand{\CX}{\mathcal{X}}

\newcommand{\CO}{\mathcal{O}}

\newcommand{\CB}{\mathcal{B}}

\newcommand{\CL}{\mathcal{L}}
\newcommand{\CP}{\mathcal{P}}
\newcommand{\CE}{\mathcal{E}}
\newcommand{\CK}{\mathcal{K}}

\newcommand{\CR}{\mathcal{R}}
\newcommand{\CS}{\mathcal{S}}
\newcommand{\CV}{\mathcal{V}}
\newcommand{\CT}{\mathcal{T}}
\newcommand{\CH}{\mathcal{H}}

\newcommand{\CG}{\mathcal{G}}
\newcommand{\CY}{\mathcal{Y}}

\newcommand{\CQ}{\mathcal{Q}}
\newcommand{\CZ}{\mathcal{Z}}
\newcommand{\CJ}{\mathcal{J}}
\newcommand{\EE}{{\mathbb E}}
\newcommand{\PP}{{\mathbb P}}
\newcommand{\YY}{{\mathbb Y}}
\newcommand{\NN}{{\mathbb N}}
\newcommand{\rr}{{\mathbb R}}
\newcommand{\LL}{{\mathbb L}}
\newcommand{\QQ}{{\mathbb Q}}
\newcommand{\DD}{{\mathbb D}}
\newcommand{\HH}{{\mathbb H}}
\newcommand{\FF}{{\mathbb F}}
\newcommand{\bS}{{\mathbb S}}
\newcommand{\KK}{{\mathbb K}}

\newcommand{\TT}{{\mathbb T}}
\newcommand{\SE}{\mathscr{E}}
\newcommand{\ST}{\mathscr{T}}
\newcommand{\SA}{\mathscr{A}}
\newcommand{\SH}{\mathscr{H}}

\newcommand{\fg}{\mathfrak{g}}
\newcommand{\fu}{\mathfrak{u}}
\newcommand{\fa}{\mathfrak{a}}
\newcommand{\fb}{\mathfrak{b}}

\newcommand{\fT}{\mathfrak{T}}
\newcommand{\fh}{\mathfrak{h}}

\newcommand{\fB}{\mathfrak{B}}

\newcommand{\wdt}{\widetilde}
\newcommand{\wdh}{\widehat}

\newcommand{\qv}[1]{\langle #1 \rangle}
\newcommand{\bqv}[1]{\big\langle #1 \big\rangle}
\newcommand{\Bqv}[1]{\Big\langle #1 \Big\rangle}

\newcommand{\tnorm}[1]{\vert\vert\vert #1 \vert\vert\vert}

\numberwithin{equation}{section}

\newtheorem{thm}{Theorem}[section]
\newtheorem{lem}[thm]{Lemma}
\newtheorem{defn}[thm]{Definition}
\newtheorem{cor}[thm]{Corollary}
\newtheorem{prop}[thm]{Proposition}
\newtheorem{rem}[thm]{Remark}

\newtheorem{ass}[thm]{Assumption}

\newcommand{\thmref}[1]{Theorem~{\rm \ref{#1}}}
\newcommand{\lemref}[1]{Lemma~{\rm \ref{#1}}}

\newcommand{\corref}[1]{Corollary~{\rm \ref{#1}}}
\newcommand{\propref}[1]{Proposition~{\rm \ref{#1}}}
\newcommand{\defref}[1]{Definition~{\rm \ref{#1}}}
\newcommand{\remref}[1]{Remark~{\rm \ref{#1}}}

\newcommand{\assmref}[1]{Assumption~{\rm \ref{#1}}}
\newcommand{\secref}[1]{Section~{\rm \ref{#1}}}

\usepackage{accents}
\usepackage{scalerel}

\newcommand{\ito}{It\^{o} }
\newcommand{\gateaux}{G\^{a}teaux }
\newcommand{\holder}{H\"{o}lder }
\newcommand{\cadlag}{c\`{a}dl\`{a}g }
\newcommand{\gronwall}{Gr\"{o}nwall }

\begin{document}
\title[Partially observed control of FSPDEs and BSDEs with jumps]{Maximum principles for partially observed
controls of forward SPDEs and backward SDEs with jumps
}

\author{Hongjiang Qian}
\address{Department of Mathematics and Statistics, Auburn University, Auburn, AL 36849}
\email{hjqian.math@gmail.com}

\author{George Yin}
\address{Department of Mathematics, University of Connecticut, CT 06269}
\email{gyin@uconn.edu}

\author{Yanzhao Cao}
\address{Department of Mathematics and Statistics, Auburn University, AL 36849}
\email{yzc0009@auburn.edu}

\author{Guannan Zhang}
\address{Computer Science and Mathematics Division, Oak Ridge National Laboratory, TN 37830}
\email{zhangg@ornl.gov}

\thanks{The research of H. Qian and Y. Cao was supported by the U.S. Department of Energy under the grant numbers DE-SC0022253, DE-SC00256, the research of G. Yin was supported in part by the National Science Foundation under grant DMS-2204240, and the research of G. Zhang was supported in part by the U.S. Department of Energy, Office of Advanced Scientific Computing Research, Applied Mathematics program under the grant ERKJ388 and ERKJ443.}


\begin{abstract}
This work establishes two versions of the Pontryagin-type maximum principles for partially observed optimal control of coupled forward stochastic partial differential equations (FSPDEs) and backward stochastic differential equations (BSDEs) with jumps in convex control domains. The FSPDE-BSDE system is driven by cylindrical Wiener processes, finite-dimensional Brownian motions, and compensated Poisson random measures. For  systems with deterministic coefficients, a direct method is employed and particular attention is
focused on
establishing the well-posedness of a singular backward SPDE with jumps. For systems with random coefficients, a Malliavin calculus approach is
developed. The main novelty here is the establishment of the well-posedness of an operator-valued SPDE with jumps, which provides a new stochastic flow representation for
linear SPDEs with jumps.
\end{abstract}
\maketitle

\section{Motivation}
The purpose of the present work is to establish Pontryagin-type maximum principles for partially observed optimal control of forward stochastic partial differential equation (FSPDEs) and backward stochastic differential equations  (BSDEs) driven by cylindrical Wiener processes, finite-dimensional Brownian motions, and compensated Poisson random measures (PRMs).

One of the motivational example is the \textit{partial information optimal harvesting}; see \cite[Section~4.1]{MMPB13}. Assume that $x(t,\xi)$ describes the density of a population (e.g. fish) at time $t\in (0,T)$ and at location $\xi\in \CO \subset \rr^n$.
Its dynamics is given by a stochastic reaction-diffusion equation with jumps
\beqa{harv}
dx(t,\xi)&\!\!\!\!\!=\big[\half \Dl x(t,
\xi)+ a(x(t,\xi))-u(t,\xi)\big]+\sg(t)x(t,\xi)dW(t) \\
&\!\!\!\! \; +\int_\rr \Theta(x(t,\xi),u(t,\xi),\al) \wdt N(dt,d\al),\quad  x(0,\xi)=x^0(\xi),\; (t,\xi)\in [0,T]\times \CO,
\eeqa
supplemented with no-flux boundary conditions on $\partial\CO$. Here $a\cd$ models the %intrinstic
intrinsic population growth. A typical choice is the logistic law $a(x)=r_a x(1-x/c_a)$, where $r_a$ is the intrinsic growth rate and $c_a>0$ is the environmental carrying capacity. The control variable $u(t,\xi)$ denotes the harvesting intensity at location $\xi$, 
assumed to take values in a prescribed admissible set reflecting regulatory constraints.
The cylindrical Wiener process $W(t)$ models  continuous environmental variability affecting reproduction and survival, 
with multiplicative coefficients of the form $\sg(t)x(t,\xi)$. The jump component driven by the compensated jump process 
$\wdt N$ captures sudden population losses induced by rare events. A biological meaningful specification is $\Theta(x,u,\al)=-\al x, \al \in (0,1)$ so that the jump instantaneously removes a random fraction $\al$ of the local population, reflecting abrupt mortality episodes.

In reality, the full spatial population field is
often not completely  observable. Instead, the decision maker has access only to partial and noisy measurements obtained from sensor networks. This is modeled through an observation process $
d Y(t) = \int_{\CO} h(\xi) x(t,\xi) d\xi + dB(t)$,
where $h\in L^2(\CO)$ represents the spatial sensitivity of the monitoring devices, $B(t)$ is an $\rr^d$-valued Brownian motion independent of $W$ and $\wdt N$. The admissible harvesting strategies are required to be adapted to the filtration generated by $Y$, reflecting the fact that decisions are based solely on partial information.

We consider a performance criterion of combining a classical accumulated cost with a recursive risk-sensitive component. For an admissible harvesting strategy $u$, the objective functional is defined by
\beq{Ju}\barray
J(u)=\EE\big[\int_0^T \int_{\CO} L(x(t,\xi),u(t,\xi))d\xi dt + \phi(x(T)) + y_0 \big],
\earray\eeq
where the running cost is $L(x,u)=-p_L u + c_L u^2/2 + \la_L /2(x-x^*)^2$ and the terminal cost takes the form $\phi(x(T))=0.5 \beta(\int_{\CO} x(T,\xi)d\xi - x_T^*)^2$. Here $p_L>0$ denotes the unit revenue from harvesting, while the quadratic term in
$u$ models increasing marginal operational and regulatory costs. The penalty term involving $x^*$ enforces sustainability by discouraging deviations from a target population level, and the terminal term prevents end-horizon depletion. The recursive component $(y_t,z_t,r_t,\ga_t)$ is characterized by the following BSDE with jumps:
\bea\!\!\!
dy_t\!=\!\big[-\eta y_t+ \vartheta_r(\|r_t\|)+ \int_{\Xi} \rho(\al)\vartheta_\ga(|\ga_t(\al)|) \pi(d\al)\big]dt+ r_t dB_t + \int_{\Xi}\ga_t(\al)\wdt N(d\al,dt)
\eea
with $y_T=0$, where
$\eta\geq 0, \rho\cd>0$ is bounded, $\pi$ denotes the $\sg$-finite measure. We choose Huber-type convex penalties
$\vartheta_r(s)=\bar \theta \dl_{r}^2(\sqrt{1+(s/\dl_r)^2} -1)$, and $\vartheta_\ga(s)=\dl_\ga^2(\sqrt{1+(s/\dl_\ga)^2} -1)$ with parameters $\bar \theta>0$ and $\dl_r,\dl_\ga>0$. These functions are $C^1$, convex, satisfy $\vartheta_r(0)=\vartheta_\ga(0)$, and have at most linear growth.

The resulting optimization problem is to minimize  $J(u)$ in \eqref{Ju} over all admissible harvesting policies adapted to the observation filtration.
While the harvesting model above serves as a representative application, stochastic reaction-diffusion equations with jumps such as \eqref{harv} arise in many other contexts, including environmental pollution dynamics describing the evolution of spatially distributed chemical concentrations; see, for instance, \cite[Chapter~19]{PZ07}.

\section{Prelude and notation}
\label{sec:intro}
We begin with two real separable Hilbert spaces
$S_1$ and $S_2$, and
denote by $\CL(S_1,S_2)$ the space of  bounded linear operators, by $\CL_1(S_1,S_2)$ the Banach space of trace-class operators,
and by $\CL_2(S_1,S_2)$ the Hilbert space of Hilbert-Schmidt operators from $S_1$ to $S_2$. When $S_1=S_2$, we simplify the notation to $\CL(S_1), \CL_1(S_1)$, and $\CL_2(S_1)$. For Banach spaces $S$ and $E$, we say a mapping $\Upsilon:S \to E$ is of class $\CG^1(S,E)$ if it is
\gateaux differentiable and its gradient $\nabla \Upsilon: S \to \CL(S,E)$ is strongly continuous (i.e., continuous in the strong operator topology).

Let $H$ be
a separable Hilbert space with inner product $\qv{\cdot,\cdot}_H$ and induced norm $|\cdot|_H$. Consider a stochastic basis $(\Omega,\CF,\{\CF_t\}_{t\in [0,T]},\PP)$ satisfying the usual conditions.
On the basis, define an $H$-valued cylindrical Wiener process $W_t$, which corresponds to space-time white noise, an $\rr^d$-valued standard Brownian motion $Y_t$, and a Poisson random measure
$N$ on $\Xi \times \rr^+$. Here $\Xi$ is a measurable space equipped with the Borel $\sigma$-field $\CB(\Xi)$ and a $\sigma$-finite measure $\pi$. We assume $W, Y,$ and $N$ are mutually independent. The compensated Poisson random measure is defined as $\wdt N(\SA, dt) := N(\SA, dt) - \pi(\SA) dt$, which is a martingale for all $\SA \in \mathcal{B}(\Xi)$ such that $\pi(\SA) < \infty$.

For any $\kappa>0$ and $s\in [0,T]$,
denote by $
L_\CP^\kappa([s,T]\times \Omega; H)=:
\CH_\kappa([s,T])$ the Banach space of $H$-valued progressively measurable process $X$ such that $\EE\int_s^T |X_t|_H^\kappa dt <\infty$, and $\HH_\kappa([s,T])$
the subspace of $H$-valued progressively measurable process $X$ such that $\EE\sup_{t\in [s,T]}|X_t|_H^\kappa<\infty$. When $s=0$, the notation $\mathcal{H}_\kappa([0, T])$ and $\mathbb{H}_\kappa([0, T])$ reduce
to $\mathcal{H}_\kappa(T)$ and $\mathbb{H}_\kappa(T)$, respectively. Furthermore, $\FF_\CP^2([s,T];H)$ denotes the space of $H$-valued $\CF_t$-predictable processes $X(\omega,t,\al)$ on $\Omega\times[0,T]\times \Xi$ such that $\EE\int_s^T\int_\Xi |X(\omega,t,\al)|_H^2 \pi(d\al)dt<\infty$. In what follows,
$|\cdot|$, $\qv{\cdot,\cdot}$ denote
the norm and inner product when the corresponding space is clear from the context; otherwise, a subscript is added for clarity. For an operator $\Phi$,  $\Phi^*$ denotes its adjoint. For any $0\leq s\leq t \leq T$,
define the product space $\Xi_s^t:=\Xi\times [s,t]$. For simplicity, we set $\Xi_T:=\Xi_0^T$. Henceforth,
$C$ is used as
a generic positive constant whose values may
change for different usage.

\subsection{Formulation}
Consider the following controlled FSPDEs and BSDEs with correlated Gaussian and jump
processes
\beq{sys}
\left\{\barray
\!\! dx_t\ad\!\!= [A x_t + F(x_t, u_t)]dt+ G_1(x_t, u_t)dW_t \\
\aad + \sum_{j=1}^d G_2^j(x_t,u_t)dB_t^j+\int_{\Xi} \Theta(x_t,u_t,\al) \wdt N(d\al, dt) \\
\!\! dy_t \ad\!\! = - \int_{\Xi} g(x_t,y_t, u_t, z_t, r_t,\ga_t(\al))\pi(d\al) dt + z_t dW_t + r_t dY_t +\int_{\Xi}\ga_t(\al) \wdt N(d\al,dt) \\
\!\! x_0\ad\!\! = x, \quad y_T = f(x_T),
\earray\right.
\eeq
where $A$ is the infinitesimal generator of a $C_0$-contraction semigroup $\{e^{tA}\}_{t\geq 0}$ of bounded linear operators, and $F,G_1,G_2^j,\Theta$ and $g$ are suitable drift and diffusion coefficients to be specified later. The $C_0$-contraction semigroup is assumed to guarantee
the applicability of maximal inequality for  stochastic convolutions with respect to compensated Poisson random measures. The control process $u$ takes values in a convex set $\CU$ contained in a separable Banach space $U$. The quintuplet $(x_t,y_t,z_t,r_t,\ga_t\cd)$, taking values in $H\times \rr^d \times \CL_2(H;\rr^d) \times \rr^{d\times d} \times L^2(\Xi;\rr^d)$, is the state process with initial state $x \in H$ and terminal condition $f(x_T)$
for a suitable function $f:H \to \rr^d$.
The process $B_t^{j}$ is a real-valued stochastic process depending on the control $u$ that will be defined later; for notational simplicity, this dependence is suppressed. In what follows,
we adopt the Einstein summation convention over finite repeated indices such that the summation symbol $\sum_{j=1}^d$ in \eqref{sys} is omitted.

Suppose that the process $(x,y,z,r,\ga\cd)$ is 
only partially observable.
That is, we have access to an observation process $Y$, which is governed by the following SDE:
\beq{ob}
dY_t = h(t,x_t,u_t) dt + dB_t, \quad Y_0 = 0,
\eeq
where $h:[0,T]\times H \times U \to \rr^d$ is a given continuous mapping. Let $\CF_t^Y:=\sg\{Y_s: s\leq t\}$ denote
the filtration generated by $Y$. The set of admissible controls is defined by
$
\CU_{\text{ad}}\!:=\!\{u:[0,T]\times \Omega\to U\!:u\text{ is } \CF_t^Y \text{-adapted and } \sup_{t\in [0,T]}\EE[\|u_t\|_{\CU}^\kappa] <\infty, \forall\, \kappa \in \NN \}.
$

\begin{ass}\label{ass}
{\rm
We assume the following conditions throughout the paper.
\begin{itemize}
\item[\rm{(H1)}] $F: H\times U  \to H$ is a map of class $\CG^1(H\times U, H)$ with bounded gradient $(\nabla_x F, \nabla_u F)$ on $H\times U$. That is, there exists a constant $C>0$ such that
\bea\!\!\!\!\!\!
|\nabla_x F(x,u)|_{\CL(H)}+ |\nabla_u F(x,u)|_{\CL(U,H)} \leq C \text{ and } |F(0,u)| \leq C, \; \forall\, x\in H, u\in U.
\eea
\item[\rm{(H2)}] $G_1:H \times U \to \CL(H)$ satisfies
$e^{sA} G_1(x,u)\in \CL_2(H)$ for all $s\geq 0, x\in H, u\in U$, and the map $(x,u)\mapsto e^{sA} G_1(x,u)\in \CG^1(H\times U, \CL_2(H))$. Besides, there exist constants $C>0$ and $\vartheta\in [0,1/2)$ such that $\forall\, x\in H,u\in U$, $|e^{sA} G_1(0,u)|_{\CL_2(H)} \leq C s^{-\vartheta}$ and
\bea
|\nabla_x[e^{sA} G_1(x,u)]|_{\CL(H,\CL_2(H))} + |\nabla_u[e^{sA}G_1(x,u)]|_{\CL(U,\CL_2(H))} \ad \leq C s^{-\vartheta}.
\eea

\item[\rm{(H3)}] $G_2^j: H\times U \to H$ is a map of class $\CG^1(H\times U, H)$ for each $j=1,\dots, d$, with bounded gradient $(\nabla_x G_2^j, \nabla_u G_2^j)$ on $H\times U$ and $|G_2^j(0,u)| \leq C,\forall\ u\in U$.

\item[\rm{(H4)}] There exists an orthonormal
basis $\{e_i\}_{i\in \NN}\in H$ such that for all $i\in \NN$ and all $u\in U$, the map $x\mapsto G_1(x,u)e_i$ is of $\CG^1(H,H)$  such that there exists a constant $C>0$, $|\nabla_x[G_1(x,u)e_i]|_{\CL(H)} \leq C,\forall\, i\in \NN, x\in H, u\in U$.

\item[\rm{(H5)}] For any $x,w \in H$, the map $u\mapsto G_1(x,u)w$ is of class $\CG^1(U,H)$ and there exists a constant $C>0$ such that $
|\nabla_u [G_1(x,u)w] v|_H \leq C|w|_H |v|_U, \; \forall\, x,w\in H$ and $u, v\in U$.

\item[\rm{(H6)}] For each $\al\in \Xi$, $\Theta(\cdot,\cdot,\al): H\times U \to H$ is class of $\CG^1(H\times U, H)$ such that for all $\kappa \geq 2$, $\max(|\nabla_\iota \Theta(x,u,\cdot)|_{L^2(\Xi;\pi)}, |\nabla_\iota \Theta(x,u,\cdot)|_{L^\kappa(\Xi;\pi)})$ are bounded for $\iota=x,u$ and $\forall\, x\in H,u\in U$. Moreover, $|\Theta(0,u,\cdot)|_{L^2(\Xi;\pi)}^2 + |\Theta(0,u,\cdot)|_{L^\kappa(\Xi;\pi)}^\kappa \leq C$ for all $u\in U$ and some constant $C>0$.

\item[\rm{(H7)}] $g: H\times U \times \rr^d\times \CL_2(H;\rr^d)\times \rr^{d\times d} \times L^2(\Xi;\rr^d) \to \rr^d$ is of class %of
$\CG^1$ in variables $(x,u,z,\ga)$ and continuously differentiable in $(r,y)$  with uniformly bounded derivatives.

\item[\rm{(H8)\footnotemark}]
The mapping $h(t,\cdot,\cdot): H \times U \to \rr^d$ is of $\CG^1(H\times U, \rr^d)$ for each $t\in [0,T]$, bounded together with its \gateaux derivative.

\item[(H9)] The initial value $x\in H$ satisfies $\EE|x|^\kappa <+\infty$ and $f:H\to \rr^d$ belongs to $\CG^1(H,\rr^d)$ with uniformly bounded derivatives.
\end{itemize}
\footnotetext{While (H8) assumes $h$ is bounded, \cite{WWX13} generalized this to linear growth in Euclidean space using approximation methods. Extending this to
solutions of SPDEs
requires further care and is reserved for future study.}
}
\end{ass}
Assumptions (H1)-(H6) ensure well-posedness of the forward SPDE system with jumps and its corresponding first-order variational equation. Specifically,
%while
(H2) encodes the smoothing property of the semigroup to diffusion coefficients, which is essential for handling the white noise and for establishing additional regularity of the adjoint process.
In addition,
(H7) guarantees the well-posedness of the backward system, (H8) imposes regularity of the observation function, and (H9) specifies integrability and regularity conditions on the initial and terminal data. An example without jumps satisfying our assumptions can be found in \cite{FHT18}.

By substituting \eqref{ob} into \eqref{sys}, we obtain
\beq{sys1}\left\{
\barray\!\!
dx_t & \!\!\!\!= [A x_t + (F-G_2^j h^j)(x_t, u_t)]dt+ G_1(x_t, u_t)dW_t \\
& + G_2^j(x_t,u_t)dY_t^j+ \int_{\Xi}\Theta(x_t,u_t,\al)\wdt N(d\al,dt) \\
\!\! dy_t & \!\!\!\!= - \int_{\Xi} g(x_t,y_t, u_t, z_t, r_t,\ga_t(\al)) \pi(d\al) dt + z_t dW_t + r_t dY_t+\int_{\Xi} \ga_t(\al) \wdt N(d\al,dt) \\
\!\! x_0 \ad \!\!\!\! =x, \quad y_T = f(x_T).
\earray
\right.
\eeq
Under \assmref{ass}, since
FSPDE-BSDE system \eqref{sys1} is decoupled, there exists a unique solution $(x^u, y^u, z^u, r^u,\ga^u\cd)$
for any $u\in \CU_{\text{ad}}$. Specifically, $x^u$ is understood as a unique \cadlag mild solution satisfying
\beqa{xu-mild}
x_t^u & = e^{tA}x + \int_0^t e^{(t-s)A} (F-G_2^j h^j)(x_s^u, u_s)ds + \int_0^t e^{(t-s)A} G_1 (x_s^u,u_s) dW_s \\
&\quad + \int_0^t e^{(t-s)A} G_2^j(x_s^u, u_s)dY_s^j+ \int_0^t \int_{\Xi} e^{(t-s)A}\Theta(x_s^u, u_s,\al) \wdt N(d\al,ds)
\eeqa
such that $x^u\in \HH_\kappa(T)$, and $y^u$ is the unique solution to the BSDE in \eqref{sys1} satisfying
\beqa{bsde-jump}
y_t^u & = f(x_T^u)+ \int_{\Xi_t^T}g(x_s^u,u_s, y_s^u, z_s^u, r_s^u,\ga_u(\al))\pi(d\al) ds \\
& \quad - \int_t^T z_s^u dW_s - \int_t^T r_s^u dY_s -\int_t^T \ga_s^u(\al)\wdt N(d\al,ds)
\eeqa
such that $(y^u,z^u,r^u,\ga^u\cd)\in L^2([0,T]\times \Omega; \rr^d \times \CL_2(H;\rr^d)\times \rr^{d\times d}\times L^2(\Xi;\rr^d))$. For the existence and uniqueness of SPDEs with jumps of \eqref{xu-mild}, we refer to Marinelli et al. \cite[Theorem 2.4]{MPR10} for details; see also Peszat and Zabczyk \cite{PZ07} and Kotelenez \cite{Kot84}. For the existence and uniqueness of BSDE with jumps of \eqref{bsde-jump}, see Tang and Li \cite{TL94} and Situ \cite{Sit97}. Our BSDE in \eqref{sys1} is also driven by a cylindrical Wiener process; however, the existence and
uniqueness
proofs hold as $z^u$ being restricted to be Hilbert-Schmidt.
Let us now introduce
\bea\!\!\!
\rho^u_t\!:= \exp\big\{\int_0^t h(s,x_s^u,u_s) dY_s- \half \int_0^t |h(s, x_s^u,u_s)|^2 ds \big\},\, B_t^j\!:=Y_t^j -\int_0^t h^j(s,x_s^u,u_s)ds.
\eea
From the It\^{o} formula, $\rho^u$ satisfies the following SDE
\beq{eq-rho}
d\rho^u_t = \rho_t^u h(t, x^u_t, u_t) dY_t, \quad \rho^u_0=1.
\eeq
Under Assumption (H8), the process $\rho_t^u$ is an $\mathcal{F}_t$-martingale. We can define a new probability measure $\mathbb{Q}^u$ on $\mathcal{F}_t$ by the Radon-Nikodym derivative such that $d\mathbb{Q}^u = \rho^u_t d\mathbb{P}$. By Girsanov's theorem, $W$ and $B$ are an $H$-valued cylindrical Wiener process and a standard $d$-dimensional Brownian motion on $(\Omega, \mathcal{F}, \{\mathcal{F}_t\}_{t \geq 0}, \mathbb{Q}^u)$, respectively, and they remain to be mutually independent. We denote the expectation with respect to the measure $\mathbb{Q}^u$ by $\mathbb{E}^u[\cdot]$.

Motivated by \cite{WWX13}, we consider a general cost functional defined by
\beq{cost}
J(u)=\EE^u \bigg[\int_0^T\int_{\Xi} L(t,x^u_t, u_t, y_t^u, z_t^u, r_t^u, \ga_t^u(\al)) \pi(d\al)dt + \phi(x_T^u)+ \psi(y^u_0) \bigg].
\eeq

\begin{ass}\label{ass:cost} We assume that the map $L: [0,T]\times H \times U \times \rr^d  \times \CL_2(H;\rr^d) \times \rr^{d\times d} \times L^2(\Xi;\rr^d) \to \rr$ and the terminal cost functional $\phi:H\to \rr$ are G\^{a}teaux differentiable with respect to $(x,u,y,z,r,\ga)$ and $x$, respectively, such that
\bea
\EE^u \big[\int_0^T\int_{\Xi} |L(x_t^u,u_t, y_t^u,z_t^u, r_t^u,\ga_t^u(\al)) | \pi(d\al) dt + |\phi(x_T^u)|+ |\psi(y_0^u)|\big]<\infty.
\eea
The mapping $\psi: \rr^d\to \rr$ is continuously differentiable.
\end{ass}

The optimal control problem for the partially observed  FSPDE-BSDE system with jumps is to find an admissible control $\wdh u\in \CU_{\text{ad}}$ such that
\beq{oc}
J(\wdh u)=\inf_{u\in \CU_{\text{ad}}}J(u).
\eeq
Using a change of measure, we observe that $J(u)$ in \eqref{cost} can be rewritten as
\beqa{J-PP}
J(u)=\EE \big[\int_0^T\int_{\Xi} \rho^u_t L(t,x_t^u, y_t^u,z_t^u, r_t^u, u_t,\ga_t^u(\al))\pi(d\al) dt + \rho_T^u \phi(x_T^u)+ \psi(y^u_0)\big].
\eeqa
The optimal control problem \eqref{oc} is thus equivalent to minimizing  the functional \eqref{J-PP} over $\CU_{\text{ad}}$ in the probability space $(\Omega,\CF,\CF_t,\PP)$, subject to \eqref{sys1} and \eqref{eq-rho}.

Let $\wdh u$ be the optimal control of \eqref{oc} and $(\wdh x, \wdh y, \wdh z, \wdh r, \wdh \ga\cd)$ be the corresponding optimal state trajectory.
We aim
to obtain two versions
of stochastic maximum principles (SMPs) for $\wdh u$ using both a direct method
and a Malliavin calculus approach. Throughout the paper, we set $\wdh \EE:= \EE^{\wdh u}$ and $\wdh \QQ:= \QQ^{\wdh u}$.

The
formulation above for
the partially observed optimal control (POOC) problem follows the classical framework of Bensoussan \cite{Ben83} and Tang \cite{Tan98}.
In this setting, the control $u$ is adapted to the observation $Y$, while $Y$ itself is defined \textit{a priori} as a Brownian motion on probability space $(\Omega,\CF,\CF_t,\PP)$, thereby ensuring its independence
with the control $u$. Partial observation optimal control
problems in Euclidean spaces have been extensively studied over the past several decades, with broad applications in economics, physics, and engineering; see, for instance, \cite{Fle68,Ben92,Par82,FP82,Tan98} and the references therein. For related problems involving recursive utility or BSDEs,
we refer to \cite{WW09,WWX13,ZS23}. In contrast, the corresponding
problems for SPDEs have received only
limited attention. Existing works include \cite{Ahm96,Ahm19,MMPB13, OPZ05,DOS18}, as well as our recent
works \cite{BCQ25,CQY25}. However, to the best of our knowledge, there are currently no results addressing
partial observation control
problems for coupled systems consisting of forward SPDEs and backward SDEs with jumps. The present work aims to fill this gap.

Compared with \cite{MMPB13} and \cite{LT23}, our
SPDE model is driven solely by a cylindrical Wiener process. Moreover, our cost functional includes both
running and terminal costs, as well as the term $\psi(y_0^u)$, which is motivated by recursive utility optimization problem \cite{LT23} and the $g$-expectation initiated by Peng \cite{Pen04}. In \cite{LT23}, the authors considered a special case of \eqref{sys1} with $G_1=z=\Theta=\gamma=0$ and $j=1$, together with $L=0$, $\phi=0$, and $\psi(y_0^u)=y_0^u$ in \eqref{cost}, under full observations. While \cite{LT23} allows for nonconvex control domains and characterizes the second-order adjoint process as the unique solution to a conditionally expected operator-valued backward stochastic integral equation, the driving noise in their SPDE is restricted to a one-dimensional Brownian motion.
It is suggested in \cite{FHT18} that
their results may be extended to the case of a cylindrical Wiener process.
Such a generalization is highly nontrivial and technically challenging, particularly in
non-convex
domains.

We summarize our contributions of this work as follows.
\begin{itemize}
  \item[(1)] We formulate a general partially observed optimal control problem for coupled FSPDE-BSDE systems driven  by cylindrical Wiener processes, finite-dimensional Brownian motions, and compensated Poisson random measures simultaneously. In particular, the cylindrical Wiener process is not assumed to
  have
  trace-class covariance operators,
  % thus giving
 allowing for space-time
  white noise and
  leading to
  a \textit{singular} backward SPDE (BPSDE) with jumps in the characterization of %the
  its adjoint process.
  \item[(2)] We
  extend
   the approximation method developed in \cite{FHT18} to
   establish
   the well-posedness of a broad class of
   of \textit{singular} BSPDEs with jumps.
  \item[(3)] For
  systems with random coefficients treated
  in \secref{sec:Mal}, we present a general framework to establish a \textit{stochastic flow} representation for solutions of linear SPDEs with jumps. To the best of our knowledge, such result is new even without jumps. Our approach
  relies
  on
  the theory of stochastic integration in Schatten-class
  operator spaces, which are M-type 2 Banach spaces; see \cite{NVW08,NVW07} for details.
% \textcolor{red}{reference}
The method is
   inspired by \cite{GP24} but allows for weaker % conditions
   assumptions
   and incorporates jump processes.
    %noise.
    We note that the techniques of \cite{MMPB13} are not applicable in our setting, as their analysis is restricted to equations driven by finitely many $\rr^d$-valued  Brownian motions.
  \item[(4)] Our work provides a three-fold generalization of existing results. First, it extends the finite-dimensional frameworks of \cite{WWX13,ZS23} to infinite-dimensional systems. Second, it generalizes \cite{LT23} by allowing
  cylindrical Wiener processes and compensated
  Poisson random measures as driving noises.
  Finally, when
  the backward SDE component $y_t$ in \eqref{sys} and the observation process in \eqref{ob} are omitted, our framework yields stochastic
  maximum principles
  for optimal control of SPDEs driven by space-time white noise and compensated Poisson random measures
  thereby extending the results of \cite{FHT18}.
\end{itemize}

The rest of the paper is organized as follows. \secref{sec:direct} devotes to
establishing the stochastic
maximum principles using a
direct approach. \secref{sec:sing} establishes the existence and uniqueness of singular backward SPDEs with jumps employed in \secref{sec:direct}. Finally, \secref{sec:Mal}
establishes the maximum principle by Malliavin calculus where all coefficients of our system are $\Omega\ni \omega$-dependent, thus 
non-Markovian.

\section{Direct approach}\label{sec:direct}
In this section, we establish the maximum principle using a direct approach. For any $v\in \CU_{\text{ad}}$ and $\e>0$, we define $u^\e:=\wdh u+ \e v$ and let $(x^\e, y^\e, z^\e, r^\e,\ga^\e\cd)$ be the solution of \eqref{sys1} with respect to the control $u^\e$. Because the control domain is convex, $u^\e\in \CU_{\text{ad}}$. In what follows, for any  $\zeta=F,G_1, G_2^j$, we denote by $\wdh \zeta(s):=\zeta(\wdh x_s,\wdh u_s)$, $\nabla_\iota \wdh \zeta(s):=\nabla_\iota\zeta(\wdh x_s, \wdh u_s)$ for $\iota=x,u$; and for $\zeta=g,L,\Theta$, we denote by $\wdh \zeta(s,\al):=\zeta(s,\wdh x_s, \wdh u_s, \wdh y_s, \wdh z_s, \wdh r_s,\wdh \ga_s(\al))$,  $\nabla_{\iota}\wdh \zeta(s,\al):=\nabla_{\iota} \zeta(s,\wdh x_s, \wdh u_s, \wdh y_s, \wdh z_s, \wdh r_s, \wdh \ga_s(\al))$ for $\iota=x,u,y,z,r,\ga$.

\begin{lem}\label{lem:mom}
Let Assumptions \ref{ass} holds. For any $u\in \CU_{\text{ad}}$, there exists a constant $C>0$ such that the solution of \eqref{sys1} and \eqref{eq-rho} satisfy the following estimates: $\forall\, \kappa \geq 2$, we have $ \EE^u |\rho^u_t|^\kappa <\infty$, and
\bea\ad
\sup_{t\in [0,T]}\EE^u|x_t^u|^\kappa \leq C\Big[1+\sup_{t\in [0,T]}\EE^u |u_t|^\kappa\Big],\; \sup_{t\in [0,T]}\EE^u |y_t^u|^2 \leq C\Big[1+\sup_{t\in [0,T]} \EE^u|u_t|^2 \Big], \\
\ad \EE^u \bigg[\int_0^T |z^u_t|_{\CL_2(H;\rr^d)}^2 + |r_t^u|^2 dt+ \int_{\Xi_T} |\ga_t^u(\al)|^2 \pi(d\al)dt \bigg] \leq C\bigg(1+\sup_{t\in [0,T]} \EE^u |u_t|^2 \bigg).
\eea
\end{lem}

\begin{proof}
The estimate $x_t^u$ in $\CH_\kappa(T)$ follows from standard arguments based on the Burkholder-Davis-Gundy (BDG) inequality and maximal inequality for stochastic convolutions driven by cylindrical Wiener process and PRMs; see \cite[Theorem 2.7]{MPR10} and \cite{DZ14}. The estimate for $y^u,z^u,r^u,\ga^u$ follows from \cite{TL94} and \cite{Sit97}. The $\kappa$-th moment estimate for $\rho^u$ is not difficult since $h$ is bounded. Thus, the details are omitted.
\end{proof}

\begin{lem}\label{lem:diff}
%Let
Assume that Assumption \ref{ass} holds. For any $v\in\CU_{\text{ad}}$ and $\forall\, \kappa \geq 2$,
\bea\ad
\sup_{t\in[0,T]} \wdh \EE|x_t^\e-\wdh x_t|^\kappa \leq C\e^\kappa, \quad \sup_{t\in [0,T]}\wdh \EE|y_t^\e-\wdh y_t|^2 \leq C \e^2, \quad \sup_{t\in [0,T]} \wdh \EE|\rho^\e_t-\wdh \rho_t|^2 \leq C \e^2, \\
\ad \wdh \EE \Big [\int_0^T \|z^\e_t-\wdh z_t\|_{\CL_2(H; \rr^d)}^2 +|r_t^\e- \wdh r_t|^2 dt +\int_{\Xi_T} |\ga_t^\e(\al)- \wdh \ga_t(\al)|^2 \pi(d\al)dt \Big] \leq C\e^2.
\eea
\end{lem}

\begin{proof}
The proof is omitted as it follows
the method in \lemref{lem:mom}. The primary distinction lies in applying \assmref{ass} to bound the coefficient differences between $(x_t^\e,u_t^\e)$ and $(\wdh x_t, \wdh u_t)$, utilizing the boundedness of their G\^{a}teaux derivatives.
\end{proof}

\subsection{First variations} We now consider the regular dependence on the parameter $\e$ for the solution $(x^\e,y^\e, z^\e, r^\e, \ga^\e\cd)$. Let us introduce
\beq{X1}\barray
dx_t^{1}&\!\!\!\! = \{A x_t^{1}+ [\nabla_x F(\wdh x_t, \wdh u_t)- \nabla_x G_2^j(\wdh x_t, \wdh u_t) h^j(t,\wdh x_t,\wdh u_t) \\
& \qquad \qquad\quad  - G_2^j(\wdh x_t,\wdh u_t)\nabla_x h^j(t,\wdh x_t,\wdh u_t)] x_t^1\} dt \\
& \; + \{[\nabla_u F(\wdh x_t,\wdh u_t)-\nabla_u G_2^j(\wdh x_t,\wdh u_t) h^j(t,\wdh x_t,\wdh u_t)-G_2^j(\wdh x_t,\wdh u_t) \nabla_u h^j(t,\wdh x_t,\wdh u_t)] v_t \}dt  \\
&\; + \{\nabla_x G_1(\wdh x_t, \wdh u_t) x_t^1 + \nabla_u G_1(\wdh x_t,\wdh u_t) v_t \} dW_t \\
& \; + \{\nabla_x G_2^j(\wdh x_t, \wdh u_t) x_t^1 + \nabla_u G_2^j(\wdh x_t,\wdh u_t)v_t \} dY_t^j \\
& \; + \int_{\Xi} \big\{\nabla_x \Theta(\wdh x_t, \wdh u_t, \al) x_t^1 + \nabla_u \Theta(\wdh x_t, \wdh u_t, \al) v_t \big\} \wdt N(d\al,dt),
\earray\eeq
\beq{y1}\barray
-dy^1_t &\!\!\!\!= \int_{\Xi} \big[\nabla_x g(t,\wdh x_t,\wdh u_t, \wdh y_t,\wdh z_t, \wdh r_t, \wdh \ga_t(\al)) x_t^1 + \nabla_u g(t,\wdh x_t,\wdh u_t, \wdh y_t, \wdh z_t, \wdh r_t, \wdh \ga_t(\al)) v_t \\
& \qquad + \nabla_y g(t,\wdh x_t,\wdh u_t, \wdh y_t, \wdh z_t, \wdh r_t, \wdh \ga_t(\al)) y_t^1   +\nabla_z g(t,\wdh x_t,\wdh u_t, \wdh y_t, \wdh z_t, \wdh r_t, \wdh \ga_t(\al)) z_t^1 \\
& \qquad+ \nabla_r g(t,\wdh x_t,\wdh u_t, \wdh y_t, \wdh z_t, \wdh r_t, \wdh \ga_t(\al)) r_t^1  \\
& \qquad + \nabla_\ga g(t,\wdh x_t,\wdh u_t, \wdh y_t, \wdh z_t, \wdh r_t, \wdh \ga_t(\al)) \ga_t^1(\al) \big] \pi(d\al) dt \\
& \; - z_t^1 dW_t - r_t^1 dY_t  - \int_{\Xi} \ga_t^1(\al) \wdt N(d\al,dt),
\earray\eeq
with $x_0^1 = 0$ and $y_T^1 = \nabla_x f(\wdh x_T) x_T^1$. Moreover, we consider
\beqa{rho1}
d\rho_t^1 =\big\{\rho_t^1 h(t,\wdh x_t, \wdh u_t)+ \rho_t \big[\nabla_x h(t, \wdh x_t, \wdh u_t) x_t^1 +\nabla_u h(t, \wdh x_t,\wdh u_t) v_t\big] \big\}dY_t, \quad \rho_0^1 = 0.
\eeqa

Define $\La:=\rho^1 \wdh \rho^{-1}$, where $\wdh{\rho}^{-1}$ denotes the inverse of $\wdh \rho$ satisfying 
\eqref{eq-rho} with $(x_t^u,u_t)$ replaced by $(\wdh x_t, \wdh u_t)$. Applying the \ito formula yields
\beq{eq-Ga}\barray
d\La_t \ad= [\nabla_x h(t,\wdh x_t,\wdh u_t) x_t^1+\nabla_u h(t,\wdh x_t, \wdh u_t)v_t ] (dY_t - h(t,\wdh x_t, \wdh u_t)dt) \\
\ad = [\nabla_x h(t,\wdh x_t, \wdh u_t) x_t^1 +\nabla_u h(t,\wdh x_t,\wdh u_t)v_t ] dB_t, \quad \La_0=0.
\earray\eeq

\begin{thm}\label{thm:x-rho-y1}
Let Assumption \ref{ass} hold. The linear SPDE with jumps \eqref{X1} admits a unique \cadlag mild solution $x^1\in \HH_2(T)$;
equations \eqref{y1} and \eqref{rho1}
(resp.)
admit unique strong solutions
such that $y^1\in C([0,T];L^2(\Omega;\rr^d))$ and $(z^1,r^1,\ga^1\cd) \in L_\CP^2(\Omega\times[0,T];\CL_2(H;\rr^d)\times \rr^{d\times d} \times L^2(\Xi;\rr^d))$. Moreover, we have $
\EE|x^1_t|^\kappa < +\infty$ and $\EE|\rho_t^1|^\kappa <+ \infty$ for any $\kappa \geq 2$.
\end{thm}

\begin{proof}
The proof for the existence and uniqueness of $x^1$ follows from \thmref{thm:sup-norm}, which will be presented in the next section; see also \cite{MPR10} for details.
The results for that of $(y^1,z^1,r^1,\ga^1)$ and $\rho^1$ are well-known since all coefficients in \eqref{y1} and \eqref{rho1} are bounded. Thus, the details are omitted.
\end{proof}

As a consequence, we have the following expansion of the cost.

\begin{prop}\label{prop:var-cost}
We have $J(u^\e)=J(\wdh u)+ \e I(v)+o(\e)$, where $I(v)$ is given by
\bea
I(v)\ad\!\! := \wdh \EE \big[\nabla_x \phi(\wdh x_T)x_T^1 + \nabla_y \psi(\wdh y_0) y_0^1 + \phi(\wdh x_T) \La_T \big]  \\
\aad + \wdh \EE\int_{\Xi_T} \Big[ \La_t \wdh L(t,\al) + \nabla_u \wdh L(t,\al)v_t + \nabla_x \wdh L(t,\al) x_t^1 + \nabla_y \wdh L(t,\al) y_t^1  \\
\ad\qquad\qquad + \nabla_z \wdh L(t,\al) z_t^1 + \nabla_r \wdh L(t,\al)r_t^1 + \nabla_\ga \wdh L(t,\al)\ga_t^1(\al) \Big] \pi(d\al) dt.
\eea
\end{prop}

\begin{proof}
Define $\wdt x_t^{1,\e}:=\e^{-1}(x_t^\e-\wdh x_t)- x_t^1$. By \lemref{lem:diff} and \thmref{thm:x-rho-y1}, we have
$|\wdt x^{1,\e}|_{\HH_2(T)} \to 0$ as $\e\to 0$.  Similarly, define $\wdt \rho_t^{1,\e}, \wdt y_t^{1,\e}, \wdt z_t^{1,\e}, \wdt r_t^{1,\e}$, and $\wdt \ga_t^{1,\e}(\al)$. Then,
$|\wdt \rho^{1,\e}|_{L^2(\Omega;C([0,T];\rr))} +|\wdt y^{1,\e}|_{C([0,T];L^2(\Omega;\rr^d))} \to 0$ and $|\wdt z^{1,\e}|_{L^2(\Omega \times [0,T]; \CL_2(H;\rr^d))} + |\wdt r^{1,\e}|_{L^2(\Omega\times [0,T]; \rr^{d\times d})}+ |\wdt \ga^{1,\e}\cd|_{L^2(\Omega; L^2([0,T]\times \Xi; \rr^d))}$ goes to zero, where the measure on $L^2([0,T]\times \Xi;\rr^d)$ is $\pi(d\al)dt$. For $\iota=x,u,y,z,r,\ga$, define
$\nabla_\iota \wdh L(t,\al,\la\e):= \nabla_\iota L(t,\wdh x_t+\la\e (x_t^1+\wdt x_t^{1,\e}), \wdh u_t+\la\e v_t, \wdh y_t+ \la \e(y_t^1+\wdt y_t^{1,\e}), \wdh z_t +\la \e (z_t^1 + \wdt z_t^{1,\e}), \wdh r_t +\la \e (r_t^1 + \wdt r_t^{1,\e}), \wdh \ga_t(\al)+\la \e (\ga_t^1(\al)+  \wdt \ga_t^{1,\e}(\al)))$.

From \eqref{J-PP}, a change of measure
yields
\beqa{var-cost}\ad
J(u^\e)-J(\wdh u) =\EE \bigg\{\rho_T^\e \big[\phi(x_T^\e)-\phi(\wdh x_T) \big] + (\rho_T^\e- \wdh \rho_T) \phi(\wdh x_T) + \psi(y_0^\e)- \psi(y_0)\\
\aad\qquad\qquad\qquad\quad + \int_{\Xi_T} \big[(\rho_t^\e- \wdh \rho_t) \wdh L(t,\al) +\rho_t^\e \big(\wdh L(t,\al,\e) - \wdh L(t,\al) \big)\big]\pi(d\al) dt \bigg\}\\
\ad = \e\, \wdh \EE \bigg[\phi(\wdh x_T) \La_T + \nabla_x \phi(\wdh x_T) x_T^1 + \nabla_y \psi(y_0) y_0^1 +\int_0^T \La_t \wdh L(t,\al) dt \bigg]+o(\e) \\
\aad \quad +\e\, \wdh \EE \int_{\Xi_T} \int_0^1 \Big[\nabla_x \wdh L(t,\al,\la\e)(x_t^1+ \wdt x_t^{1,\e})\Big] d\la\, \pi(d\al) dt \\
\aad \quad + \e\, \wdh \EE \int_{\Xi_T} \int_0^1 \Big[\nabla_y \wdh L(t,\al,\la\e)(y_t^1+ \wdt y_t^{1,\e})\Big]d\la\, \pi(d\al) dt \\
\aad \quad +\e\, \wdh \EE \int_{\Xi_T} \int_0^1 \Big[\nabla_z \wdh L(t,\al,\la\e)(z_t^1+ \wdt z_t^{1,\e})\Big]d\la\,\pi(d\al) dt \\
\aad \quad +\e\, \wdh \EE \int_{\Xi_T}  \int_0^1 \Big[\nabla_r \wdh L(t,\al,\la\e)(r_t^1+ \wdt r_t^{1,\e})\Big]d\la\, \pi(d\al) dt \\
\aad \quad +\e\, \wdh \EE \int_{\Xi_T}  \int_0^1 \Big[\nabla_\ga \wdh L(t,\al,\la\e)(\ga_t^1(\al)+ \wdt \ga_t^{1,\e}(\al)) \Big]d\la\, \pi(d\al) dt \\
\aad \quad +\e\, \wdh \EE \int_{\Xi_T} \int_0^1 \nabla_u \wdh L(t,\al,\la\e) v_t d\la\, \pi(d\al) dt.
\eeqa
By the continuity and boundedness of $\nabla_\iota \wdh L$, the dominated convergence theorm implies
\bea&
\wdh \EE \int_{\Xi_T} \int_0^1 \nabla_x \wdh L(t,\al,\la\e) x_t^1 d\la\, \pi(d\al) dt \to \wdh \EE \int_{\Xi_T} \nabla_x \wdh L(t,\al) x_t^1 \pi(d\al)dt \\
& \wdh \EE \int_{\Xi_T} \int_0^1 \nabla_x \wdh L(t,\al,\la \e ) \wdt x_t^{1,\e} d\la\, \pi(d\al) dt  \to 0.
\eea
Similar results apply
to other terms in \eqref{var-cost}. The proof is complete.
\end{proof}

\subsection{Adjoint equations}\label{sec:adjoint} We now fix an orthonormal basis $\{e_i\}_{i\in \NN}$ in $H$ such that
\assmref{ass} (H4) holds.
For any $\chi \in H$, define \beq{Ki}\barray
K_i(t)\chi := \nabla_x[G_1(\wdh x_t, \wdh u_t)e_i]\chi=\nabla_x[\wdh G_1(t)e_i]\chi \text{ and }
\Ga_t \chi :=\nabla_u[G_1(\wdh x_t, \wdh u_t)\chi]v_t.
\earray\eeq

\begin{ass}\label{ass1} Assume that
$e^{tA}\in \CL_2(H)$ for all $t>0$ and there exists constants $C>0$ and $\vartheta <1/2$ such that $|e^{tA}|_{\CL_2(H)} \leq C t^{-\vartheta}, \forall\, t\in (0,T]$.
\end{ass}

\begin{rem}\label{rem:est-Ki}\rm{
By \assmref{ass1}, we have $|K_i(t)|_{\CL(H)} \leq C$ and (H2) implies
\bea\ad\!\!\!\!\!
\sum_{i=1}^\infty |e^{tA} K_i(s) \chi|_H^2 =\sum_{i=1}^\infty |\nabla_x [e^{tA}\wdh G_1(s)e_i]\chi |_H =|\nabla_x \big[e^{tA} \wdh G_1(s) \chi\big]|_{\CL_2(H)}^2 \leq C t^{-2\vartheta} |\chi|_H^2.
\eea
for all $t>0, s\geq 0, \chi\in H$; see \cite[Remark 2.2 and p. 265]{FHT18}.
}
\end{rem}

Define $\beta_t^i:=\qv{e_i, W_t}_H, i=1,2,\dots$.
Then $\{\beta^i\}$ is a family of independent real-valued  standard Brownian motions. We now introduce the adjoint equations on the probability space $(\Omega,\CF,\CF_t,\wdh \QQ)$. The first adjoint processes $\ell_t$ satisfies
\beqa{ellt}
\!\!\! d\ell_t\ad\!\!\! = \!\! \int_{\Xi}\!\big[
\nabla_y \wdh g(t,\al)\ell_t - \nabla_y \wdh L(t,\al)\big] \pi(d\al) dt +\! \int_{\Xi}\! \big[\nabla_z^* \wdh g(t,\al) \ell_t - \nabla_z \wdh L(t,\al)\big] \pi(d\al) dW_t \\
\aad + \int_{\Xi} \big[\nabla_r \wdh g(t,\al)\ell_t - \wdh h(t) \ell_t - \nabla_r \wdh L(t,\al) \big] \pi(d\al) dB_t \\
\aad + \int_{\Xi} \big[\nabla_\ga \wdh g(t,\al) \ell_{t} - \nabla_\ga \wdh L(t,\al)\big] \wdt N(d\al,dt), \quad \ell_0=-\nabla_y \psi(y_0).
\eeqa
The adjoint equations for the unknown processes $(P,Q_{1,t},Q_{2,t}^j,Q_{3,t}\cd)$ are formally given by
\beqa{Pt}\barray
\!\!\!\! -dP_t &\!\!\!\! \disp = \Big\{ A^* P_t + \big[\nabla_x \wdh F(t)-\wdh G_2^j(t) \nabla_x \wdh h^j(t)\big] P_t + \sum_{i=1}^\infty K_i^*(t) Q_{1,t} e_i \\
& \!\!\!\! + \nabla_x^* \wdh G_2^j(t) Q_{2,t}^j + \nabla_x^* \wdh h^j (t) q_{2,t}^j \Big\} dt -\! \int_{\Xi} \big[ \nabla_x^* \wdh g(t,\al) \ell_t + \nabla_x \wdh L(t,\al) \big] \pi(d\al) dt\\
&\!\! \disp - \sum_{i=1}^\infty Q_{1,t}e_i d\beta_t^i - Q_{2,t}^j dB_t^j - \int_{\Xi} Q_{3,t}(\al) \wdt N(d\al,dt)
\earray\eeqa
with terminal condition $P_T = \nabla_x \phi(\wdh x_T)-\nabla_x^* f(\wdh x_T)\ell_T$. Finally, the auxiliary adjoint process
$(p,q_{1,t}, q_{2,t}^j, q_{3,t}\cd)$
satisfies
\beqa{pt}\ad\!\!\!
- dp_t = \int_{\Xi} \wdh L(t,\al) \pi(d\al)dt - q_{1,t}dW_t -q_{2,t}^j dB_t^j - \int_{\Xi} q_{3,t}(\al) \wdt N(d\al,dt), \; p_T = \phi(\wdh x_T).
\eeqa

Formally, $P_t, p_t, \ell_t$ are interpreted as adjoint
processes
for $x_t, \rho_t, y_t$, respectively.  Eq. \eqref{pt} has a unique solution $(p,q_{1},q_2^j,q_3\cd)\in L^2([0,T]\times \Omega; \rr\times \CL_2(H\;\rr)\times \rr^{1\times d} \times L^2(\Xi;\rr))$ whose proof is standard; see \cite{TL94}.
In addition, Eq. \eqref{ellt} is a forward equation with linear coefficients and therefore admits a unique solution.

For equation \eqref{Pt}, we
emphasize
that the series $\sum_{i=1}^\infty K_i^*(t)Q_{1,t}e_i$ appearing in \eqref{Pt} is, in general, not convergent, even when $Q_{1,t}$ is Hilbert-Schmidt.
This lack of convergence reflects the \textit{singular} nature of \eqref{Pt}, which originates from the cylindrical Wiener noise.
We therefore adopt
the idea of
the approximation approach of
\cite{FHT18} to deal with this term. Let us recall the definition of a mild solution of \eqref{Pt} in \cite[Definition 4.1]{FHT18}.

\begin{defn}\label{def:mild-P}\rm{
The quadruple $(P,Q_1,Q_2^j,Q_{3}\cd)$ with $P\in L_{\CP}^2(\Omega\times [0,T], H)$, $Q_{1}\in L_{\CP}^2(\Omega\times[0,T], \CL_2(H))$, $Q_{2}^j \in L_{\CP}^2(\Omega\times[0,T], H)$, and $Q_{3}\in \FF_\CP^2(\Omega\times [0,T];H)$ is a mild solution of equation \eqref{Pt}, if (i) 
%the sequence 
$\mathbb{K}^m(t):=\sum_{i=1}^m (T-t)^{\vartheta} K_i^*(t)Q_{1,t}e_i,\, t\in [0,T]$ converges weakly in $L_\CP^2(\Omega\times [0,T];H)$; (ii) for any $t\in [0,T]$, we have $\wdh \QQ$-a.s.,
\bea
P_t \ad\!\!\! = e^{(T-t)A^*} \big[\nabla_x \phi(\wdh x_T)-\nabla_x f(\wdh x_T)\ell_T \big]  +\sum_{i=1}^\infty \int_t^T e^{(s-t)A^*} K_i^*(s) Q_{1,s} e_i d\beta_s^i \\
\aad\!\! +\int_t^T\! e^{(s-t)A^*} \Big\{\Big[\nabla_x \wdh F(s)-\wdh G_2^j(s) \nabla_x \wdh h^j(s)\Big] P_s + \nabla_x^* \wdh G_2^j(s) Q_{2,s}^j + \nabla_x^* \wdh h^j(s) q_{2,s}^j\Big\} ds \\
\aad\!\!\! +\int_t^T \int_{\Xi} e^{(s-t)A^*} \Big[-\nabla_x^* \wdh g(s,\al) \ell_s + \nabla_x \wdh L(s,\al)\Big] \pi(d \al) ds  \\
\aad\!\!\! -\sum_{i=1}^\infty \int_t^{T} e^{(s-t)A^*} Q_{1,s}e_i d\beta_s^i -\int_t^T e^{(s-t)A^*} Q_{2,s}^j dB_s^j \\
\aad\!\!\! - \int_t^T \int_{\Xi} e^{(s-t)A^*} Q_{3,t}(\al) \wdt N(d\al,ds).
\eea
}
\end{defn}

\begin{prop}\label{prop:exist}
Under Assumptions \ref{ass} and \ref{ass1}, there exists a unique mild solution $(P,Q_{1},Q_2^j,Q_3\cd)$ for $j=1,\dots, d$ to \eqref{Pt} in the sense of \defref{def:mild-P}. Moreover, $Q_{1,t}\in\CL_1(H), d\wdh \QQ \otimes dt$-a.s. satisfying $\wdh \EE\int_0^T (T-t)^{-2\vartheta} |Q_{1,t}|_{\CL_1(H)}^2 dt <\infty$.
\end{prop}

\begin{proof}
The proof follows from \thmref{thm:ext-P}, detailed in the next section.
\end{proof}

Define
\beq{def-OR}
\wdh O_t :=\nabla_x \wdh F(t)- \wdh G_2^j(t) \nabla_x \wdh h^j(t), \quad \wdh R_t :=\nabla_u \wdh F(t)- \wdh G_2^j(t) \nabla_u \wdh h^j(t).
\eeq
Recall the definition of $K_i, \Ga$ in \eqref{Ki}. Substituting \eqref{ob} into \eqref{X1} expresses $x^1$ in terms of $B^j$:
\beqa{X1-re}\barray
\!\!\!\!\! dx_t^1\!\!\!\!\! & = \disp \big[Ax_t^1 + \wdh O_t x_t^1 +\wdh R_t v_t \big] dt + \sum_{i=1}^\infty K_i(t) x_t^1 d\beta_t^i + \sum_{i=1}^\infty \Ga_t e_i d\beta_t^i \\
\!\!\! &\;+\big[\nabla_x\wdh G_2^j(t)x_t^1+ \nabla_u \wdh G_2^j(t)v_t \big] dB_t^j+ \int_{\Xi} \big[\nabla_x \wdh \Theta(t,\al) x_t^1 + \nabla_u \wdh \Theta(t,\al)v_t\big] \wdt N(d\al,dt).
\earray\eeqa
The main gradient of deriving maximum principles is the following duality relation.

\begin{prop}\label{prop:dual}
With previous assumptions and notation, suppose that $R^\dag:[0,T]\times \Omega \to H$ and $\Ga^\dag: [0,T]\times \Omega \to \CL_2(H)$ are progressively measurable and bounded. Let $\CX$ be the unique mild solution of the following equation
\beqa{CX}\barray
\!\!\!\!\! d\CX_t &\!\!\!\! \disp = \big[A\CX_t + \wdh O_t \CX_t + R^\dag_t \big] dt + \sum_{i=1}^\infty K_i(t) \CX_t \, d \beta_t^i + \sum_{i=1}^\infty \Ga_t^\dag e_i \, d\beta_t^i \\
&\!\!\!+ \big[\nabla_x\wdh G_2^j(t) \CX_t + \nabla_u \wdh G_2^j(t)v_t \big] dB_t^j+ \int_{\Xi} \big[\nabla_x \wdh \Theta(t,\al) \CX_t + \nabla_u \wdh \Theta(t,\al)v_t\big] \wdt N(d\al,dt).
\earray\eeqa
Then we have
\beq{dual}\barray\ad\!\!\!\!\!\!
\wdh \EE \int_{\Xi_T}\!\!\! \Big[ \qv{P_t,R_t^\dag}\!+ \! \qv{Q_{1,t},\Ga_t^\dag}_{\CL_2(H)} \!+\!\qv{Q_{2,t}^j, \nabla_u \wdh G_2^j(t)v_t}\!+\! \qv{Q_{3,t}(\al), \nabla_u \wdh \Theta(t,\al)v_t} \Big]\pi(d\al) dt\\
\aad \!\!\!\!\!\! = \wdh \EE \Big\{\bqv{\nabla_x \phi(\wdh x_T)-\nabla_x^* f(\wdh x_T)\ell_T, \CX_T}_H \\
\aad \quad + \int_{\Xi_T} \bqv{\nabla_x \wdh L(t,\al)-\nabla_x^* \wdh g(t,\al)\ell_t + \nabla_x^* \wdh h^j(t) q_{2,t}^j, \CX_t}_H\, \pi(d\al) dt \Big\}.
\earray\eeq
\end{prop}

\begin{proof}
The result follows from \corref{cor:conv-Pm} detailed in the next section, in which we put $s=0,\eta=\chi=0$, $ \CR_t^\dag= R_t^\dag, \CR_{t}^{2,j} =\nabla_u \wdh G_2^j(t)v_t, \CR_t^3(\al)=\nabla_u \wdh \Theta(t,\al) v_t$, and $\eta=\nabla_x \phi(\wdh x_T)-\nabla_x^* f(\wdh x_T)\ell_T$, $\CJ_t= \nabla_x L(t,\al)-\nabla_x^* \wdh g(t,\al) \ell_t +\nabla_x^* \wdh h^j(t) q_{2,t}^j$.
\end{proof}

\subsection{Stochastic maximum principle}\label{sec:max}
Next, we present the main result.

\begin{thm}\label{thm:main}
Let Assumptions \ref{ass}, \ref{ass:cost}, and \ref{ass1} hold. Let $(\wdh x, \wdh u, \wdh y, \wdh z, \wdh r, \wdh \ga)$ be the optimal state proesss. For any $v\in \CU_{\text{ad}}$, we have
\bea\ad
\wdh \EE \Big\{\bqv{\wdh R_t(v-\wdh u_t), P_t} + \textnormal{Tr}\Big[Q_{1,t}^* (\nabla_u \wdh G_1(t) (v-\wdh u_t))\Big] \\
\aad\quad + \bqv{Q_{2,t}^{j}, \nabla_u \wdh G_2^j(t) (v-\wdh u_t)} +\int_\Xi \bqv{Q_{3,t}(\al), \nabla_u \wdh \Theta(t,\al)(v-\wdh u_t)} \pi(d\al) \\
\aad\quad + \bqv{\nabla_u \wdh L(t,\al)+\nabla_u^* \wdh h^j(t)\, q_{2,t}^j - \nabla_u^* \wdh g(t,\al)\ell_t , v-\wdh u_t}_{\CU} \,\Big|\, \CF_t^Y \Big\} \geq 0,
\eea
where $(P,Q_1,Q_{2}^j, Q_3\cd), (p,q_1,q_{2}^j,q_3\cd)$, and $\ell$ are solutions of \eqref{Pt}, \eqref{pt}, and \eqref{ellt}, respectively.
\end{thm}

\begin{proof}
Applying the \ito formula to $p\,\La$ and $\qv{\ell, y^1}_{\rr^d}$, respectively, we have
\beq{dual-rho}\barray
\wdh \EE \big[ \La_T \phi(\wdh x_T)+\int_{\Xi_T} \La_t \wdh L(t,\al) \pi(d\al)dt \big]=\wdh \EE\int_0^T \bqv{\nabla_x^* \wdh h^j(t)\, q_{2,t}^j, x_t^1} dt,
\earray\eeq
and
\beqa{dual-y}\barray & \!\!\!\!\!
\wdh \EE \big[\bqv{\ell_T, y_T^1}_{\rr^d} - \bqv{\ell_0, y_0^1}_{\rr^d} \big] = \wdh \EE \big[\bqv{\ell_T, \nabla_x f(\wdh x_T) x_T^1} +\bqv{\nabla_y \psi(y_0), y_0^1} \big] \\
&\!\!\!\!\! = -\wdh \EE \int_{\Xi_T} \big[\nabla_y \wdh L(t,\al) y_t^1 + \nabla_z \wdh L(t,\al) z_t^1 + \nabla_r \wdh L(t,\al) r_t^1+\nabla_\ga \wdh L(t,\al)\ga_t^1(\al) \big]\, \pi(d\al) dt \\
&\!\!\!\!\! \quad -\wdh \EE \int_{\Xi_T} \bqv{\nabla_u \wdh g(t,\al)v_t+ \nabla_x \wdh g(t,\al) x_t^1, \ell_t}_{\rr^d} \, \pi(d\al) dt.
\earray\eeqa
We next claim that the following duality holds:
\beqa{dual-X} &\!\!\!\!\!\!\!
\wdh \EE \int_0^T \qv{P_t, \wdh R_t v_t } + \qv{Q_{1,t},\Ga_t}_{\CL_2(H)}+\qv{Q_{2,t}^j, \nabla_u G_2^j(t)v_t} dt \\
&\!\!\!\!\! \quad  +\wdh \EE \int_{\Xi_T}\qv{Q_{3,t}(\al), \nabla_u \wdh \Theta(t,\al)v_t }  \pi(d\al)dt  \\
&\!\!\!\!\!\!\! =\wdh \EE \qv{\nabla_x \phi(\wdh x_T)- \nabla_x^* f(\wdh x_T)\ell_T, x_T^1}\! -\! \wdh \EE \int_{\Xi_T}\! \qv{\nabla_x^* \wdh g(t,\al)\ell_t, x_t^1} \pi(d\al) dt \\
&\!\!\!\!\! \quad +\wdh \EE \int_{\Xi_T} \qv{\nabla_x \wdh L(t,\al) + \nabla_x^* \wdh h^j(t) q_{2,t}^j, x_t^1} \pi(d\al) dt.
\eeqa
Since $v$ is taken as a bounded process, it follows that $\wdh R_tv_t= [\nabla_u \wdh F(t)-\wdh G_2^j(t)\nabla_u h^j(t)]v_t$ is bounded. However, \propref{prop:dual} cannot be applied directly, since $\Ga$ in \eqref{X1-re} for $x^1$ is not a bounded $\CL_2(H)$-valued process. The duality formula \eqref{dual-X} needs to be established
%via
using
an approximation argument based on \propref{prop:dual}.

Denote by $\mathfrak{T}_n$ the orthogonal projection on $H$ onto span $\{e_1,\dots, e_n\}$, and define, for $\chi \in H$, $\Ga_t^n \chi =\nabla_u[G_1(\wdh x_t, \wdh u_t) (\fT_n \chi)]v_t = \sum_{i=1}^n \nabla_u[G_1(\wdh x_t, \wdh u_t) e_i] v_t \qv{\chi,e_i}_H$. Then each $\Ga_t^n$ is a bounded $\CL(H)$-valued  process, and hence also bounded in $\CL_2(H)$. Let $x^{1,n}$ denote the mild solution of \eqref{X1-re} with $\Ga_t$ replaced by $\Ga_t^n$. By \propref{prop:dual}, the duality relation \eqref{dual} holds with $\Ga_t^\dag,R_t^\dag$, and $\CX$ replaced by $\Ga_t^n, \wdh R_t v_t$, and $x_t^{1,n}$, respectively, for all $t\in [0,T]$. Consequently, letting $n\to \infty$ and invoking \propref{prop:exist}, we can obtain the desired duality \eqref{dual-X}. The convergence argument is standard and therefore omitted, as it closely follows the proof of \cite[pp. 267--268]{FHT18}.

By \propref{prop:var-cost}, substituting \eqref{dual-rho}, \eqref{dual-y}, and \eqref{dual-X} into \eqref{var-cost} yields
\bea & \!\!\!
\wdh \EE \int_0^T \qv{\wdh R_t v_t, P_t}+ \text{Tr}\big[Q_{1,t}^* (\nabla_u \wdh G_1(t) v_t) \big] dt +  \qv{Q_{2,t}^{j} \nabla_u \wdh G_2^j(t)  v_t} +\qv{\nabla_u^* \wdh h^j(t) q_{2,t}^j, v_t}dt \\
& \!\!\! + \wdh \EE \int_{\Xi_T} \qv{Q_{3,t}(\al), \nabla_u \wdh \Theta(t,\al)v_t} + \qv{\nabla_u \wdh L(t,\al)- \nabla_u^* \wdh g(t,\al) \ell_t, v_t} \pi (d\al)dt \geq 0.
\eea
Since $\wdh u$ is optimal, we have $J(u^\e)-J(\wdh u)\geq 0$ and the proof can be concluded by standard arguments based on localization, the Lebesgue differentiation theorem and the observation that $v$ is $\CF^Y$-measurable random variable. This completes the proof.
\end{proof}

\section{Singular backward SPDEs with jumps}\label{sec:sing}
This section proves Propositions \ref{prop:exist} and \ref{prop:dual}. Inspired by the approximation scheme in \cite{FHT18} but in contrast to \cite{FHT18},
our effort focuses on treatment of
jumps
from compensated Poisson random measures.

\subsection{Linear FSPDEs with jumps and auxiliary estimates}
Let us consider the following linear forward SPDE with jumps:
\beqa{CY}\barray
d\CY_t & \disp = \Big[A \CY_t + \CQ_t^1 \CY_t +  \CR_t^\dag \Big] dt + \sum_{i=1}^\infty K_i(t) \CY_t d\beta_t^i + \sum_{i=1}^\infty K_i(t)\chi_t \, d\beta_t^i + \sum_{i=1}^\infty \Ga_t^\dag e_i \,d\beta_t^i \\
&  \quad + \big[\CQ_t^{2,j} \CY_t + \CR_t^{2,j}\big] d B_t^j+ \int_{\Xi} \big[\CQ_t^3(\al) \CY_t + \CR_t^3(\al) \big] \wdt N(d\al,dt), \quad \CY_s=y,
\earray\eeqa
together with its approximating equations, for $m,n \in \NN^+$ and initial $\CY_s^{m,n} =y$,
\beqa{CY-m}\barray\ad\!\!\!\!\!\!\!\!\!\!\!\!
d \CY_t^{m,n} = \Big[A \CY_t^{m,n}+ \CQ_t^1  \CY_t^{m,n} + \CR_t^\dag \Big] dt +\sum_{i=1}^n K_i(t) \CY_t^{m,n}\, d \beta_t^i + \sum_{i=1}^m K_i(t) \chi_t \, d\beta_t^i \\
\aad \;\, + \sum_{i=1}^\infty \Ga_t^\dag e_i d\beta_t^i + \big[\CQ_t^{2,j} \CY_t^{m,n} + \CR_t^{2,j} \big] dB_t^j + \int_{\Xi}\big[\CQ_t^3(\al) \CY_t^{m,n} + \CR_t^3(\al) \big] \wdt N(d\al,dt),
\earray\eeqa
where we assume $\CQ^1,\CR^\dag,\CQ^{2,j}, \CR^{2,j}, \chi:[0,T]\times \Omega \to H$ and $\Ga^\dag: [0,T]\times \Omega \to \CL_2(H)$ are progressively measurable and bounded. For $\zeta= \CQ^3, \CR^3$, $\zeta$ is assumed to be  progressively measurable and for all $t\in [0,T]$, $\max(|\zeta_t|_{L^\kappa(\Xi,\pi)}, |\zeta_t|_{L^2(\Xi,\pi)}^\kappa)$ is bounded. For $\chi=0$, $\CY_t^{m,n}$ is independent of $m$ and is denoted by $\CY_t^{n}$. To emphasize the dependence on initial datum $y$, the solution of \eqref{CY} is also denoted by $\CY^y$.

\begin{thm}\label{thm:sup-norm}
For $m,n\in \NN^+$ and $\kappa \geq 2$.
Equations \eqref{CY} and \eqref{CY-m} admit unique \cadlag mild solutions $\CY$ and $\CY^{m,n}$ in $\HH_\kappa([s,T])$, respectively. Moreover, the solution map $y\mapsto \CY^y$ is Lipschitz from $L^\kappa(\Omega,\CF,\wdh \QQ;H)$ to $\HH_\kappa([s,T])$ and we have
\beqa{mom-CY}\disp
\wdh \EE\bigg[\sup_{t\in [s,T]}|\CY_t|^\kappa \bigg] \ad \leq C\bigg[ 1+\EE|y|^\kappa + |\Ga^\dag|_{L_\CP^\infty([0,T]\times \Omega; \CL(H))}^\kappa \\
\aad\qquad\; + |\chi|^\kappa_{L_\CP^\infty([s,T]\times\Omega;H)} + \EE\Big(\int_0^T |\CR_s^\dag| ds \Big)^\kappa \bigg].
\eeqa
Furthermore, the following convergences hold in $\HH_\kappa([s,T])$ as $m, n \to \infty$: $\CY^{m,n} \to \CY^{\infty,n}, \CY^{m,n}\to \CY^{m,\infty}, \CY^{m,m}\to \CY, \CY^{\infty,n}\to \CY$, and $\CY^{m,\infty}\to \CY$.
\end{thm}

\begin{proof}
The existence for SPDE with compensated Poisson random measures
under \assmref{ass} is well-known, we omit the details but refer to \cite{MPR10}. The idea is to define a mapping $\ST$ from $\HH_\kappa([s,T])$ to itself by
\beqa{def-CT}\barray
\ST(\CY)_t &\!\!\!\! = e^{(t-s)A} \CY_0 + \int_s^t e^{(t-\tau)A} (\CQ_\tau^1 \CY_\tau +\CR_\tau^\dag) d\tau +  \int_s^t e^{(t-\tau)A}\big[\CQ_\tau^2 \CY_\tau +\CR_\tau^2 \big] dB_\tau^j\\
& \disp + \sum_{i=1}^\infty \int_s^t e^{(t-\tau)A}  K_i(\tau) \CY_\tau d\beta_\tau^i +\sum_{i=1}^\infty \int_s^t e^{(t-\tau)A}K_i(\tau)\chi_\tau d\beta_\tau^i  \\
& + \sum_{i=1}^\infty \int_s^t e^{(t-\tau)A}  \Ga_\tau^\dag e_i d\beta_\tau^i + \int_{\Xi_s^t} e^{(t-\tau)A} \big[\CQ_\tau^3(\al) \CY_\tau + \CR_\tau^3(\al)\big]  \wdt N(d\al,d\tau).
\earray\eeqa
By showing this map is a contraction under the equivalent norm of $\HH_\kappa([s,T])$, namely, $\|\CY\|^\kappa = \wdh \EE\sup_{t\in [s,T]} e^{-\kappa \varrho t} |\CY_t|_H^\kappa$ for some $\varrho>0$, the existence is then guaranteed by the Banach fixed point theorem.

We proceed to prove \eqref{mom-CY}. The boundedness of $\CQ^1$ implies
\beqa{QR-1}\ad\!\!\!\!\!\!
\wdh \EE\sup_{t\in [s,T]}\Big|\int_s^t\! e^{(t-\tau)A}( \CQ_\tau^1 \CY_\tau + \CR_\tau^\dag)d\tau \Big|^\kappa \!\leq C_T \int_s^T \wdh \EE\sup_{\sg\in [s,\tau]}|\CY_{\sg}|^\kappa d\tau + \wdh \EE\bigg[\int_s^T\! |\CR_\tau^\dag |d\tau \bigg]^\kappa.
\eeqa
The definition of $K_i(t)$ in \eqref{Ki} and stochastic factorization formula in \cite[Theorem 5.10]{DZ14} yield that there exist a $0<\dl< 1/2- \vartheta$ and a constant $c_\dl>0$ such that
\beq{fact}
\sum_{i=1}^\infty
\int_s^t\!\! e^{(t-\tau)A}  K_i(\tau)\CY_\tau d\beta_\tau^i = \int_s^t\!\! e^{(t-\tau)A} \nabla_x\wdh G_1(\tau)\CY_\tau dW_\tau= c_\dl \int_s^t\!\! (t-\tau)^{\dl-1} e^{(t-\tau)A} \bar \CY_\tau d\tau,
\eeq
where
$ \bar \CY_\tau:= \int_s^\tau (\tau-\sg)^{-\dl} e^{(\tau-\sg)A} \nabla_x \wdh G_1(\sg) \CY_\sg dW_{\sg}$. The BDG inequality implies
\beqa{bar-CY}\ad\!\!\!\!\!
\wdh \EE|\bar \CY_\tau|^\kappa  =\wdh \EE\Big|\int_s^{\tau} (\tau-\sg)^{-\dl} e^{(\tau-\sg)A} \nabla_x \wdh G_1(\sg) \CY_{\sg}dW_{\sg}\Big|^\kappa \\
\aad\!\!\!\!\! \leq C\,  \wdh \EE \Big[ \int_s^\tau (\tau-\sg)^{-2\dl} |e^{(\tau-\sg)A} \nabla_x \wdh G_1(\sg)\CY_{\sg}|_{\CL_2(H)}^2 d\sg \Big]^{\kappa/2}  \\
\aad\!\!\!\!\!\! \leq C\, \wdh \EE \Big[ \int_s^\tau (\tau-\sg)^{-2\dl -2\vartheta} |\CY_{\sg}|^2 d\sg \Big]^{\kappa/2} \leq C\, \wdh\EE \sup_{\sg \in [s,\tau]} |\CY_{\sg}|^\kappa \Big(\int_s^\tau (\tau-\sg)^{-2\dl-2\vartheta} d\sg\Big)^{\kappa/2}.
\eeqa
Consequently, the \holder inequality gives
\bea\disp
\wdh \EE\sup_{t\in [s,T]}\Big|\sum_{i=1}^\infty \int_s^t e^{(t-\tau)A} K_i(\tau)\CY_\tau d\beta_\tau^i \Big|^\kappa \ad = c_\dl^\kappa \, \wdh \EE\sup_{t\in [s,T]} \Big|\int_s^t (t-\tau)^{\dl-1}e^{(t-\tau)A} \bar \CY_\tau d\tau \Big|^\kappa \\
\aad \leq C_{\kappa, \dl,T} \int_s^T \wdh \EE\sup_{\sg\in [s,\tau]} |\CY_{\sg}|^\kappa d\tau.
\eea
Similarly, \remref{rem:est-Ki} and the boundedness of $\Ga^\dag$ imply that
\bea\disp
\wdh \EE\sup_{t\in [s,T]}\Big|\sum_{i=1}^\infty \int_s^t e^{(t-\tau)A} K_i(t)\chi_\tau d\beta_\tau^i \Big|_H^\kappa \ad \leq C|\chi|_{L_\CP^\infty([s,T]\times \Omega;H)}^\kappa,  \\
\disp \wdh \EE\sup_{t\in [s,T]}\Big|\sum_{i=1}^\infty \int_s^t e^{(t-\tau)A} \Ga_\tau^\dag  e_i d\beta_\tau^i\Big|_H^\kappa \ad \leq C |\Ga^\dag|_{L_\CP^\infty([s,T]\times \Omega;H)}^\kappa.
\eea
Moreover, the BDG inequality and the boundedness of $\CQ^2$ and $\CR^2$ yield
\bea\ad
\wdh \EE \sup_{t\in [s,T]} \Big|\int_s^t e^{(t-\tau)A}(\CQ_\tau^2 \CY_\tau + \CR_\tau^2) dB_\tau^j \Big|^\kappa \leq C_T \Big[1+\int_s^T \wdh \EE\sup_{\sg\in[s,\tau]}|\CY_{\sg}|^\kappa d\tau \Big].
\eea
For the jump term in \eqref{def-CT}, employing the maximal inequality for stochastic convolutions driven by compensated Poisson random measures (see \cite[Proposition 3.3]{MPR10} and \cite{PZ07,ZBH17}) gives
\bea\ad\!\!\!\!\!\!\!
\wdh \EE\sup_{t\in [s,T]}\Big|\int_{\Xi_s^t} e^{(t-\tau)A} [\CQ_\tau^3(\al)\CY_\tau+\CR_\tau^3(\al)] \wdt N(d\al,d\tau)\Big|^\kappa \\
\aad\!\!\!\!\!\!\! \leq C \,\wdh \EE\int_s^T\!\! \bigg[\int_{\Xi} |\CQ_\tau^3(\al)\CY_\tau+ \CR_\tau^3(\al)|^\kappa \pi(d\al) + \Big(\int_\Xi |\CQ_\tau^3(\al)\CY_\tau +\CR_\tau^3(\al)|^2 \pi(d\al) \Big)^{\kappa/2} \bigg] d\tau \\
\aad\!\!\!\!\!\!\! \leq C\,
\wdh \EE \int_s^T \Big (|\CQ_\tau^3|_{L^\kappa(\Xi;\pi)}^\kappa + |\CQ_\tau^3|_{L^2(\Xi; \pi)}^\kappa \Big) |\CY_\tau|^\kappa + \Big(|\CR_\tau^3|_{L^\kappa(\Xi;\pi)}^\kappa + |\CR_\tau^3|_{L^2(\Xi;\pi)}^\kappa \Big) d\tau  \\
\aad\!\!\!\!\!\!\! \leq C\Big(1+\int_s^T \EE\sup_{\sg\in [s,\tau]}|\CY_{\sg}|^\kappa d \tau \Big).
\eea
Combining above estimates and applying Gr\"{o}nwall's inequality yield \eqref{mom-CY}. The convergence results follow directly. The Lipschitz continuity of the solution map $y\mapsto \CY^y$ is established by an argument analogous to that of \cite[Theorem 2.3]{MPR10}. The details are omitted. The proof is complete.
\end{proof}

\begin{prop}\label{prop:mom-2}
Suppose that Assumptions \ref{ass} and \ref{ass1} hold, we obtain
\beq{sup-L2}
\sup_{t\in [s,T]}\wdh \EE|\CY_t|^2 \leq C\,\wdh \EE \bigg[1+|y|^2 + \bigg(\int_s^T |\CR^\dag_\tau| d\tau \bigg)^2+ |\chi|_{L_\CP^\infty([s,T]\times \Omega;H)}^2 + \int_s^T |\Ga_t^\dag|^2_{\CL_2(H)}dt \bigg]
\eeq
and for any $s\leq t \leq T$,
\beqa{CY-L2}
\wdh \EE|\CY_t|_H^2 \ad \leq C\,\wdh \EE \bigg[1+|x|^2+ \bigg(\int_s^t |\CR_\tau^\dag| d\tau \bigg)^2+ \int_s^t (t-\tau)^{-2\vartheta} |\chi_\tau|^2 d\tau \\
\aad\qquad\quad  + \int_s^t (t-\tau)^{-2\vartheta} |\Ga_\tau^\dag|_{\CL(H)}^2 d\tau \bigg].
\eeqa
\end{prop}

\begin{proof}
We first note that
\beq{e-Ga}
|e^{(t-\tau)A} \Ga_\tau^\dag|_{\CL_2(H)} \leq |e^{(t-\tau)A}|_{\CL_2(H)} |\Ga_\tau^\dag|_{\CL(H)}\leq C (t-\tau)^{-\vartheta}|\Ga_\tau^\dag|_{\CL(H)}.
\eeq
To prove \eqref{sup-L2}, we consider the mild formulation of $\CY_t$ in \eqref{def-CT} with $\CY_t$ substituted for $\ST(\CY)_t$. For the jump term in \eqref{def-CT}, the \ito isometry for stochastic integrals driven by compensated Poisson random measures together with
\eqref{e-Ga} (with $\Ga^\dag$ replaced by $\CQ_\tau^3(\al)\in \CL(H)$) yield that
\bea\ad
\wdh \EE \Big|\int_s^t \int_{\Xi} e^{(t-\tau)A}[\CQ_\tau^3 (\al)\CY_\tau + \CR_\tau^3(\al)] \wdt N(d\al, d\tau) \Big|^2 \\
\aad \leq C\, \wdh \EE\int_s^t \int_{\Xi} |e^{(t-\tau)A}[\CQ_\tau^3(\al) \CY_\tau + \CR_\tau^3(\al)] |^2  \pi(d\al)d\tau \\
\aad \leq C  \int_s^t (t-\tau)^{-2\vartheta} \wdh \EE \Big[|\CQ_\tau^3|_{L^2(\Xi;\pi)}^2 |\CY_\tau|^2 \Big]  d\tau + \wdh \EE \int_s^t |\CR_\tau^3|_{L^2(\Xi;\pi)}^2 d\tau \\
\aad \leq C\bigg(1+\int_s^t (t-\tau)^{-2\vartheta} \wdh \EE|\CY_\tau|^2 d\tau \bigg).
\eea
For the rest of terms on the r.h.s. of \eqref{def-CT}, from \assmref{ass1} and \remref{rem:est-Ki}, we can apply similar arguments as above to obtain
%, and thus obtain
\beqa{E-CY-2}\!\!
\wdh \EE|\CY_t|^2 \ad\!\!\! \leq C\, \wdh \EE \bigg\{1+|y|^2 + \bigg(\int_s^t |\CR_\tau^\dag |d\tau \bigg)^2 + \int_s^t (t-\tau)^{-2\vartheta} |\chi_\tau|^2 d\tau  \\
\aad\qquad\quad + \int_s^t \big|e^{(t-\tau)A} \Ga_\tau^\dag \big|_{\CL_2(H)}^2 d\tau + \int_s^t (t-\tau)^{-2\vartheta} \EE|\CY_\tau|^2 d\tau \bigg\} \\
\aad\!\!\!\!\!\! \leq C\,\wdh \EE\bigg[ 1+|y|^2 + \bigg(\int_s^T |\CR_\tau^\dag| d\tau \bigg)^2 + |\chi|_{L^\infty}^2+\int_s^T|\Ga_\tau^\dag|_{\CL_2(H)}^2 d\tau \\
\aad \qquad\quad + \int_s^t  (t-\tau)^{-2\vartheta} \EE|\CY_\tau|^2 d\tau \bigg].
\eeqa
By a variant of \gronwall inequality in \cite[Lemma 7.1.1]{Hen81}, we can obtain the estimate \eqref{sup-L2}. The Lipschitz continuity of the solution map $y\mapsto \CY^y$ in $\CH_2([s,T])$ follows a similar argument with above; see also \cite[Theorem 2.3]{MPR10} for details.

To obtain estimate of \eqref{CY-L2} with the norm $|\Ga|_{\CL(H)}$, we utilize \eqref{e-Ga} again. Define
\bea\!\!
\mathfrak{u}(t) := \wdh \EE|\CY_t|^2,\ \mathfrak{a}(t) := C\, \wdh \EE \big(|\chi_t|^2+|\Ga_t|_{\CL(H)}^2 \big),\ \fb(t) : =C\, \wdh \EE \big[1+|y|^2+\big(\int_s^t |\CR_\tau^1|d\tau \big)^2\big].
\eea
Plugging \eqref{e-Ga} into the first inequality of \eqref{E-CY-2}, using the
above notation, we obtain
\bea
\fu(t)\leq \fb(t)+\int_s^t (t-\tau)^{-2\vartheta} \fa(\tau)d\tau+ C\, \int_s^t (t-\tau)^{-2\vartheta} \fu(\tau)d\tau.
\eea
Set $\mathfrak{h}(t):=\fb(t)+\int_s^t (t-\tau)^{-2\vartheta}\fa(\tau)d\tau$. Using Gr\"{o}nwall's
inequality in \cite{Hen81} again gives
\bea
\fu(t)\leq \fh(t)+C \int_s^t(t-\tau)^{-2\vartheta} \mathfrak{h}(\tau)d\tau.
\eea

Let us compute the second term on the r.h.s. of the above inequality:
\bea
\int_s^t (t-\tau)^{-2\vartheta} \fh(\tau)d\tau & =\int_s^t (t-\tau)^{-2\vartheta} \fb(\tau)d\tau + \int_s^t (t-\tau)^{-2\vartheta} \int_s^{\tau} (\tau-\sg)^{-2\vartheta} \fa(\sg)d\sg d\tau \\
&  \leq \fb(t)\int_s^t (t-\tau)^{-2\vartheta}d\tau + \int_s^t \fa(\sg) \int_{\sg}^t (t-\tau)^{-2\vartheta}(\tau-\sg)^{-2\vartheta} d\tau d\sg \\
& \leq C \fb(t) +C \int_s^t (t-\sg)^{-2\vartheta} \fa(\sg)d\sg,
\eea
where the second line follows from the
changing
order of integration, and the third line follows from
the estimate $
\int_{\sg}^t (t-\tau)^{-2\vartheta}(\tau-\sg)^{-2\vartheta} d\tau \leq C (t-\sg)^{-2\vartheta}$. Consequently, \eqref{CY-L2} holds, and the proof is complete.
\end{proof}

\subsection{Existence of singular BSPDEs with jumps}
Consider the following class of backward SPDEs with jumps
of the form
\beqa{CZ}\barray
-d Z_t &\disp = \Big[A^* Z_t + \CV_t Z_t + \sum_{i=1}^\infty K_i^*(t) Q_{1,t} e_i + \CE_t^j Q_{2,t}^j+ \CJ_t \Big] dt \\
& \quad -\sum_{i=1}^\infty Q_{1,t} e_i d\beta_t^i - Q_{2,t}^j dB_t^j -\int_{\Xi} Q_{3,t}(\al) \wdt N(d\al,t),\quad Z_T = \eta,
\earray\eeqa
where we assumed $\eta\in L^2(\Omega,\CF_T,\wdh\QQ,H)$, $\CJ\in L_\CP^2(\Omega\times[0,T];H)$, and $\CV,\CE^j: [0,T]\times \Omega \to H$ are progressively measurable and bounded.
To handle the series $\sum_{i=1}^\infty K_i^*(t) Q_{1,t}e_i$ in the drift of \eqref{CZ}, we employ a truncation argument.  For any $n\in \NN^+$, consider
\beqa{CZ-n}\barray
-d Z_t^n & \disp = \Big[A^* Z_t^n + \CV_t Z_t^n + \sum_{i=1}^n K_i^*(t) Q_{1,t}^n e_i + \CE_t^j Q_{2,t}^{j,n}+ \CJ_t \Big] dt \\
& \quad -\sum_{i=1}^\infty Q_{1,t}^n e_i d\beta_t^i - Q_{2,t}^{j,n} dB_t^j -\int_{\Xi} Q_{3,t}^n (\al) \wdt N(d\al,t), \quad Z_T^n=\eta.
\earray\eeqa

The following proposition establishes
the existence and uniqueness of solutions to the semilinear backward stochastic evolution equation (BSEE) with jumps for \eqref{CZ-n}.

\begin{prop}\label{prop:ext-Pn}
There exists a unique mild solution $(Z^n, Q_{1,t}^n, Q_{2}^{j,n}, Q_{3}^n\cd)$ of \eqref{CZ-n} such that $Z^n\in L_\CP^2(\Omega\times [0,T]; H)$, $Q_1^n \in L_\CP^2(\Omega\times [0,T];\CL_2(H))$, $Q_2^{j,n}\in L_\CP^2(\Omega\times [0,T];H)$, and $Q_3^n \in \FF_\CP^2([0,T]; H)$ satisfying
\beqa{mild-Zn}\!\!\!
Z_s^n\ad\!\!\! = e^{(T-s)A^*}\eta+ \int_s^T e^{(t-s)A^*} \Big[\CV_t Z_t^n +\sum_{i=1}^n K_i^*(t) Q_{1,t}^n e_i + \CE_t^j Q_{2,t}^{j,n}+ \CJ_t \Big]dt \\
\aad \!\!\! -\sum_{i=1}^\infty \int_s^T\!\! e^{(t-s)A^*} Q_{1,t}^n e_i d\beta_t^i -\!\! \int_s^T\!\! e^{(t-s)A^*} Q_{2,t}^{j,n} dB_t^j -\! \int_{\Xi_s^T} \! \! e^{(t-s)A^*} Q_{3,t}^n(\al) \wdt N(d\al,dt).
\eeqa
Moreover, the following duality
%formula
holds
\beqa{dual-n}\ad\!\!
\wdh \EE\qv{Z_s^n,y} + \wdh \EE\int_s^T \Big[\qv{Z_t^n, \CR_t^\dag} + \sum_{i=1}^m \qv{Q_{1,t}^n e_i, K_i(t)\chi_t} + \qv{Q_{1,t}^n, \Ga_t^\dag}_{\CL_2(H)}+ \qv{Q_{2,t}^{j,n}, \CR_t^{2,j}}  \\
\aad +\int_{\Xi_s^T} \qv{Q_{3,t}^n(\al), \CR_t^3(\al) } \pi(d\al) \Big] dt  = \wdh \EE\qv{\eta, \CY_T^{m,n}} + \wdh \EE\int_s^T \qv{\CJ_t, \CY_t^{m,n}} dt
\eeqa
where $\CY_t^{m,n}$ is the solution of \eqref{CY-m}.
\end{prop}

\begin{proof}
Note that \eqref{CZ-n} has no singular drift. Hence, the existence and uniqueness of a mild solution to \eqref{CZ-n} follow directly from \cite{CH07} via successive approximation; see also \cite{OPZ05} for a variational approach. Moreover, the duality formula follows from a standard application of \ito formula (cf. \cite{HP91}). Accordingly, we omit the details.
\end{proof}

Consequently, we define a candidate solution $(Z,Q_1,Q_2^j,Q_3\cd)$ to \eqref{CZ} as the weak limit of $(Z^n,Q_1^n,Q_2^{j,n}, Q_3^n\cd)$ for $j=1,\dots,d$ in some  Hilbert space.

\begin{cor}\label{cor:conv-Pm}
We have as $n\to \infty$,
\begin{itemize}
\item[\rm{(1)}] $(Z^n,Q_1^n, Q_2^{j,n}, Q_3^n\cd)$ converges weakly to an element $(Z,Q_1,Q_2^j,Q_3\cd)$ in the space $L_\CP^2([0,T]\times \Omega;H) \times L_\CP^2([0,T]\times \Omega;\CL_2(H)) \times L_\CP^2([0,T]\times \Omega;H) \times \FF_\CP^2([0,T]; H)$;
\item[\rm{(2)}] for each $t\in [0,T]$, $Z_t^n$ converges weakly to an element $\wdt Z_t$ in $L^2(\Omega,\CF_t,\wdh \QQ,H)$.
\end{itemize}
Moreover, for $\eta\in L^2(\Omega,\CF_T,\wdh \QQ, H)$ and all $m\in \NN$, we have
\beqa{dual-Pm}
\ad
\wdh \EE\qv{\wdt Z_s,y} + \wdh \EE\int_s^T \Big[\qv{Z_t, \CR_t^\dag} + \sum_{i=1}^m \qv{Q_{1,t}e_i, K_i(t)\chi_t}+ \qv{Q_{1,t}, \Ga_t^\dag}_{\CL_2(H)} + \qv{Q_{2,t}^{j}, \CR_t^{2,j}} \\
\aad + \int_{\Xi_s^T} \qv{Q_{3,t}(\al), \CR_t^3(\al) } \pi(d\al)\Big] dt = \wdh \EE\qv{\eta, \CY_T^{m,\infty}} + \wdh \EE\int_s^T \qv{\CJ_t, \CY_t^{m,\infty}} dt,
\eeqa
where $\CY_t^{m,\infty}$ is the solution of \eqref{CY-m} with $n=\infty$. In particular, when $\chi=0$, the process $\CY^{m,n}$ does not depend on $m$, so that $\CY^{m,\infty}=\CY$ in \eqref{dual-Pm}.
\end{cor}

\begin{proof}
The proof is parallel to \cite[Corollary 4.8]{FHT18},
so the details are omitted.
\end{proof}

To proceed, we examine the regularity
of $\wdt Z_s$.

\begin{prop}\label{prop:weak-cont}
The map $s\mapsto \wdt Z_s$ from $[0,T]$ to 
$L^2(\Omega,\CF_T,\wdh \QQ,H)$
is weakly continuous.
\end{prop}

\begin{proof}
By \eqref{dual-Pm}, to examine the weakly continuity of $\wdt Z_s$, it is sufficient to consider the linear FSPDE  \eqref{CY} with $\Ga^\dag=\chi=\CR^\dag=\CR^2=\CR^3=0$, that is,
\beqa{CY-0}
d\CY_t = \big[A\CY_t + \CQ_t^1 \CY_t \big]dt + \sum_{i=1}^\infty K_i(t)\CY_t d\beta_t^i + \CQ_t^{2,j}\CY_t dB_t^j +\int_{\Xi} \CQ_t^3(\al) \CY_t \wdt N(d\al,dt)
\eeqa
with initial condition
$\CY_s=y$.
For any $y \in L^2(\Omega,\CF_s,\wdh \QQ,H)$, we denote by $\CY_t^{s,y}$ the solution of \eqref{CY-0} to emphasize its dependence on the initial datum $(s,y)$. \thmref{thm:sup-norm} implies that  \eqref{CY-0} admits a unique mild solution
\beqa{mild-CY-0}
\CY_t^{s,y}\ad = e^{(t-s)A} y+ \int_s^t e^{(t-\tau)A} \CQ_s^1 \CY_\tau^{s,y} d\tau + \sum_{i=1}^\infty \int_s^t e^{(t-\tau)A} K_i(\tau) \CY_\tau^{s,y}  d\beta_\tau^i \\
\aad + \int_s^t e^{(t-\tau)A} \CQ_\tau^{2,j} \CY_\tau^{s,y}  dB_\tau^j +  \int_{\Xi_s^t} e^{(t-\tau)A} Q_\tau^3(\al)\CY_\tau^{s,y} \wdt N(d\al,d\tau), \quad \CY_s^{s,x}=y,
\eeqa
in $\HH_2([s,T])\cap \CH_2([s,T])$ such that $\sup_{t\in [s,T]}\wdh \EE|\CY_t^{s,y}|^2 \leq C_T(1+|y|_H)$ and
\beq{CY-lip-init}
\sup_{t\in [s,T]}\wdh \EE|\CY_t^{s,y}-\CY_t^{s,y'}|^2 \leq C_T\, \wdh \EE|y-y'|, \quad \forall\, y' \in L^2(\Omega,\CF_s,\wdh \QQ,H).
\eeq
For fixed $y\in L^2(\Omega,\CF_T,\wdh\QQ,H)$, let $y_s:=\wdh \EE(y|\CF_s)$. Taking conditional expectation to \eqref{dual-Pm} gives $
\wdh \EE\qv{\wdt Z_s, y}= \wdh \EE\qv{\wdt Z_s, y_s}= \wdh \EE\qv{\eta,\CY_T^{s,x_s}} + \wdh \EE \int_s^T \qv{\CJ_t, \CY_t^{s,y_s}} dt$.

To prove the weak continuity of $\wdt Z$, it is sufficient to prove for all $t\geq s$, the map $s \mapsto \CY_t^{s,y_s}$ is continuous in the norm of $L^2(\Omega,
\CF,\wdh \QQ,H)$. For any $\sg>s$, note that \bea
\wdh \EE|\CY_t^{\sg,y_{\sg}}-\CY_t^{s,y_s}|^2 \leq \wdh \EE|\CY_t^{\sg,y_{\sg}}-\CY_t^{\sg,y_s}|^2 + \wdh \EE|\CY_t^{\sg,y_s}- \CY_t^{s,y_s}|^2 =:\YY_1 +\YY_2.
\eea
The term $\YY_1$ above is controlled by \eqref{CY-lip-init} and goes to zero since $\wdh \EE|y_s-y_{\sg}|^2 =\wdh \EE|\wdh \EE(y|\CF_{\sg})-\wdh \EE(y|\CF_s)|^2$ goes to $0$ as $\sg\to s$. For the term $\YY_2$, we note that $\EE|\CY_t^{\sg,y_s}-\CY_t^{s,y_s}|^2=\EE|\CY_t^{\sg,y_s}-\CY_t^{\sg,\CY_{\sg}^{s,y_s}}|^2 \leq \EE|\CY_{\sg}^{s,y_s}-y_s|^2$. From \eqref{mild-CY-0}, following a similar argument to that in \propref{prop:mom-2} implies
\bea\ad
\wdh \EE|\CY_{\sg}^{s,y_s}-y_s|^2 \\
\aad \leq 32\, \wdh \EE|e^{(\sg-s)A}y-y_s|_H^2 + 32\, \wdh \EE \Big|\int_s^t e^{(t-\tau)A} \CQ_\tau^1 \CY_\tau^{s,y_s} d\tau \Big|^2 \\
\aad\quad + 32\, \wdh \EE\Big|\sum_{i=1}^\infty \int_s^{\sg} e^{(\sg-\tau)A}  K_i(\tau)\CY_\tau^{s,y_s} d\beta_\tau^i \Big|^2 +\wdh \EE \Big|\int_s^{\sg} e^{(\sg-\tau)A} \CQ_\tau^{2,j} \CY_\tau^{s,y_s} dB_\tau^j\Big|^2 \\
\aad\quad + 32\, \wdh \EE \Big|\int_s^{\sg} \int_{\Xi} e^{(\sg-\tau)A} Q_{\tau}^3(\al) \CY_\tau^{s,y_s} \wdt N(d\al,d\tau) \Big|^2 \\
\ad \leq 32\, \wdh \EE|e^{(\sg-s)A} y_s - y_s|^2 + C \sup_{\tau\in [s,T]}\wdh \EE|\CY_\tau^{s,y_s}|^2 \int_s^{\sg} (\sg-\tau)^{-2\vartheta} d\tau \to 0
\eea
as $\sg\downarrow s$. For $\sg\uparrow s$, the calculation is similar. Therefore, the proof is complete.
\end{proof}

We are in
a
%the
position to prove that $\wdt Z$ and $Z$ coincide.
\begin{prop}\label{prop:wdt-Z}
$\wdt Z$ is a progressively measurable process and $\wdt Z = Z, d\wdh\QQ\otimes dt$-a.s.
\end{prop}

\begin{proof}
For the progressive measurability, fix any $t\in[0,T]$, we choose a basis $\{\upsilon_l\}_{l=1}^\infty$ of $L^2(\Omega,\CF,\wdh\QQ,H)$. Therefore, $\wdt Z_s=\sum_{l=1}^\infty \wdh \EE\qv{\wdt Z_s, \upsilon_l}_H \upsilon_l$ for all $s\leq t$. By \propref{prop:weak-cont}, $\wdh \EE\qv{\wdt Z_s,\upsilon_l}$ is continuous function in time, thus $\wdt Z$ restricted to $[0,t]$ is $\CB([0,t])\otimes \CF_t$-measurable. To prove $\wdt Z$ and $Z$ coincide, we choose $y=\chi=\Ga^\dag=0$ and any arbitrary bounded progressively measurable process $\CR^\dag$ in \eqref{CY}. By definition of $\wdt Z$, for all $t\in [0,T]$, we have $\wdh \EE\qv{Z_t^n, \CR_t^\dag}_H \to \wdh \EE\qv{\wdt Z_t, \CR_t^\dag}_H$. Moreover, the dominated convergence theorem and the measurability of $\wdt Z$ then implies $\int_0^T \wdh \EE\qv{Z_t^n, \CR_t^\dag}_H dt \to \int_0^T \wdh \EE\qv{\wdt Z_t, \CR_t^\dag}_H dt$. However, by \corref{cor:conv-Pm}, we know that $\int_0^T \wdh \EE\qv{Z_t^n, \CR_t^\dag}_H dt \to \int_0^T \wdh \EE \qv{Z_t,\CR_t^\dag}_H dt$. Thus, $\wdt Z=Z$, $d\wdh\QQ\otimes dt $-a.s.
\end{proof}

To prove the existence of a solution of the singular BSPDE \eqref{CZ}, we pass to the limit $n\to \infty$ in \eqref{mild-Zn}. For $m\in \NN$, define $\mathbb{K}^m(s):=\sum_{i=1}^m (T-s)^{\vartheta} K_i^*(s) Q_{1,s} e_i$.

\begin{lem}\label{lem:Q1-m} We have $(i)$ $\mathbb{K}^m$ converges weakly in $L_\CP^2(\Omega\times [0,T];H)$ with the limit denoted by $\sum_{i=1}^\infty (T-\cdot)^\vartheta K_i^*(\cdot) Q_{1,\cdot} e_i$. $(ii)$ The term $(T-\cdot)^\vartheta \sum_{i=1}^m K_i^*(\cdot) Q_{1,\cdot}^m e_i$ converges weakly to $(T-\cdot)^\vartheta \sum_{i=1}^\infty K_i^*(\cdot) Q_{1,\cdot} e_i$ in the space $L_\CP^2(\Omega\times [0,T];H)$.
\end{lem}
\begin{proof}
In light of
\cite[Lemma 4.11, 4.12]{FHT18}, our main effort is  devoted to the jump system. In \eqref{CY} and \eqref{CY-m}, we set $s=y=\Ga^\dag=\CR^\dag=\CQ^1=\CQ^k=\CR^k=0$ for $k=2,3$, take $n=\infty$, replace $\chi$ with $(T-\cdot)^\vartheta \chi$, and denote the corresponding mild solutions by $\CY^{\chi}$ and $\CY^{m,\chi}$, respectively. By the definition of $\mathbb{K}^m$ and the duality 
in \eqref{dual-Pm}, it follows that
\beqa{CV-chi}\disp\!\!
\wdh \EE\int_0^T\!\! \bqv{\mathbb{K}^m(t),\chi_t}dt \ad =\wdh \EE\int_0^T\!\! \Bqv{\sum_{i=1}^m (T-t)^\vartheta K_i^*(t)Q_{1,t}e_i, \chi_t}dt \\
\aad 
= \wdh \EE \bqv{\eta, \CY_T^{m,\chi}}+ \EE\int_0^T\!\! \bqv{\CJ_t,\CY_t^{m,\chi}}_H dt.
\eeqa
The estimate \eqref{CY-L2} yields $|\wdh \EE\int_0^T \qv{\mathbb{K}^m(t),\chi_t}_H dt|\leq C|\chi|_{L_\CP^2(\Omega\times[0,T];H)}$. Since bounded elements are dense in $L_\CP^2(\Omega\times [0,T];H)$, this inequality implies the sequence of $\{\KK^m\}$ is uniformly bounded in $L_\CP^2(\Omega\times[0,T];H)$. Furthermore,  \thmref{thm:sup-norm} ensures that the right-hand side of \eqref{CV-chi} converges as $m\to \infty$ for any bounded $\chi$. Consequently, $\mathbb{K}^m$ converges weakly in $L_\CP^2(\Omega\times[0,T];H)$.

To prove statement (ii),
%let
denote by $\wdt \CY_t^{m,\chi}$ and $\wdt \CY_t^{\chi}$
%denote
the solutions to Eq. \eqref{CY} and \eqref{CY-m}, respectively, under the setting that $s=y=\Ga^\dag=\CR^\dag=\CQ^1=\CQ^k=\CR^k=0$ for $k=2,3$, with $n=m$, and with $\chi$ replaced by $(T-\cdot)^\vartheta \chi$.
%The
% Statement (ii)
The conclusion
then follows by an argument analogous to %the proof of the
that used in part (i), with \eqref{CV-chi} replaced by
\bea
\wdh \EE\int_0^T \qv{(T-t)^{\vartheta} \sum_{i=1}^m K_i^*(t)Q_{1,t}^m e_i, \chi_t} dt =\wdh \EE\qv{\eta,\wdt \CY_T^{m,\chi}}+ \wdh \EE\int_0^T \qv{\wdt \CY_t^{m,\chi},\CJ_t}dt.
\eea
In view of the convergence $\wdh \EE\sup_{t\in [0,T]}|\wdt \CY_t^{m,\chi}-\wdt \CY_t^\chi|^2 \to 0$ established in \thmref{thm:sup-norm}, we conclude that the term $(T-\cdot)^{\vartheta}\sum_{i=1}^m K_i^*(\cdot)Q_{1,\cdot}^m e_i$ converges with the limit denoted by $(T-\cdot)^{\vartheta}\sum_{i=1}^\infty K_i^*\cd Q_1 e_i$. This completes the proof.
\end{proof}

The following lemma establishes the existence and uniqueness of solutions to linear BSPDEs with jumps and unbounded forcing terms. This result serves as a technical tool for proving the uniqueness of the singular BSPDE with jumps \eqref{CZ}.
% Since the proof can be adapted directly from \cite[Lemma 4.13]{FHT18} to the jump setting by using results in \cite{MPR10} to deal with stochastic convolutions driven by PRMs, we omit the details.

\begin{lem}\label{lem:bar-P}
Assume $\xi$ is a progressively measurable progress in $H$ with $\wdh \EE\int_0^T (T-t)^{2\vartheta} |\xi_t|^2 dt<\infty$. For any $n\in \NN$ and any $\eta\in L^2(\Omega,\CF_T,\wdh \QQ;H)$, there exists a unique quadruple $(\CO,\CZ_1,\CZ_2^j, \CZ_3\cd)$ with $\CO$ is progressively measurable in $H$ with \cadlag path such that $\CO \in \HH_2(T)$, $\CZ_1\in L_\CP^2(\Omega \times[0,T];\CL_2(H))$, $\CZ_2^j\in L_\CP^2(\Omega\times[0,T]; H)$, and $\CZ_3\in \FF_\CP^2([0,T];H)$ such that
\beqa{CO-mild}\!\!
\CO_s \ad\!\!\!\! = e^{(T-s)A^*} \eta +\!\! \int_s^T\!\!\! e^{(t-s)A^*} \CV_t \CO_t dt + \!\! \int_s^T\!\!\! e^{(t-s)A^*} \sum_{i=1}^n K_i^*(t) \CZ_{1,t} e_i dt   \\
\aad \!\!\!\!+ \int_s^T e^{(t-s)A^*} \xi_t dt +\int_s^T e^{(t-s)A^*} \CE_t^j \CZ_{2,t}^j dt-\sum_{i=1}^\infty\! \int_s^T\!\! e^{(t-s)A^*} \CZ_{1,t} e_i d\beta_t^i  \\
\aad \!\!\!\!-\! \int_s^T\!\! e^{(t-s)A^*} \CZ_{2,t}^jdB_t^j - \int_{\Xi_s^T}\!\! e^{(t-s)A^*} \CZ_{3,t}(\al) \wdt N(d\al,dt).
\eeqa
Moreover, let $\wdt \CY^n$ be the solution of \eqref{CY-m} with $\chi=\CR^\dag=0, m=\infty$ and $\Ga^\dag \in L_\CP^\infty(\Omega\times[0,T]; \CL_2(H))$, we have
\beqa{dual-CO}\ad
\wdh \EE\bqv{\CO_s,y}+ \wdh \EE\int_s^T \Big[\bqv{\CZ_{1,t}, \Ga_t^\dag}_{\CL_2(H)} +\bqv{\CZ_{2,t}^j, \CR_{t}^{2,t}} + \int_{\Xi} \bqv{\CZ_{3,t}(\al),\CR_t^3(\al)} \pi(d\al)\Big] dt \\
\aad = \wdh \EE\bqv{\eta, \wdt \CY_T^n}+\wdh \EE\int_s^T \bqv{(T-t)^{\vartheta}\xi_t, (T-t)^{-\vartheta} \wdt \CY_t^n}dt.
\eeqa
\end{lem}

\begin{proof}
The key idea is to apply the martingale representation theorem with respect to $W,B^j$, and $\wdt N$.
This is similar in spirit to \cite[Lemma 4.3]{FHT18} and \cite{HP90},
so
the
verbatim proof is omitted for brevity.
\end{proof}

\begin{thm}\label{thm:ext-P}
The quadruple $(Z,Q_1,Q_2^j,Q_3\cd)$ constructed in \corref{cor:conv-Pm} is the unique mild solution to the singular BSPDE with jumps \eqref{CZ}.
\end{thm}

\begin{proof}
\textbf{Step 1: Existence.} For any $s\in [0,T]$, we
observe that the mapping
$\mathfrak{g} \mapsto \int_s^T (T-t)^{-\vartheta} e^{(t-s)A^*} \fg_t dt$
defines a bounded linear functional from $L_\CP^2(\Omega\times [0,T];H)$ to $L^2(\Omega,\CF_T,\wdh\QQ,H)$ and is therefore weakly continuous. Consequently, for each fixed $s\in [0,T]$, \lemref{lem:Q1-m} implies that the sum
%\bea
$\sum_{i=1}^m \int_s^T e^{(t-s)A^*} K_i^*(t)Q_{1,t} e_i \, dt =  \int_s^T (T-t)^{-\vartheta} e^{(t-s)A^*} \mathbb{K}^m (t)\, dt$
converges weakly in $L^2(\Omega,\CF_t,\wdh \QQ, H)$, to a limit that we denote by  
$\sum_{i=1}^\infty \int_s^T e^{(t-s)A^*} K_i^*(t) Q_{1,t}e_i dt.$
Similarly, 
$\sum_{i=1}^n \int_s^T e^{(t-s)A^*} K_i^*(t) Q_{1,t}^n e_i \, dt$ converges weakly in $L^2(\Omega,\CF_T,\wdh \QQ,H)$ to $\sum_{i=1}^\infty \int_s^T e^{(t-s)A^*} K_i^*(t) Q_{1,t} e_i\, dt$. Moreover, by the weak convergence of $Q_{1,t}^n$ to $Q_{1,t}$ in $L^2(\Omega\times [0,T];\CL_2(H))$ established in \corref{cor:conv-Pm}, we also have $\sum_{i=1}^\infty \int_s^T e^{(t-s)A^*} Q_{1,t}^n e_i d \beta_t^i$ converges weakly to $\sum_{i=1}^\infty \int_s^T e^{(t-s)A^*} Q_{1,t} e_i d\beta_t^i$. Passing to the limit in \eqref{CZ-n} shows that $(Z,Q_{1},Q_2^j, Q_{3}\cd)$ is a mild solution of \eqref{CZ}.

\textbf{Step 2: Uniqueness.} Let $(Z,Q_1,Q_2^j, Q_3\cd)$ and $(Z',Q'_1,Q_2'^{j}, Q_3'\cd)$ be two solutions of \eqref{CZ}. Define their difference as $\bar Z = Z-Z'$, $\bar Q_1=Q_1-Q'_1$, $\bar Q_2^j= Q_2^j-Q_2'^{j}$, and $\bar Q_3\cd=Q_3\cd-Q'_3\cd$. Then, \eqref{CZ} implies
\bea
\bar Z_s \ad = \int_s^T e^{(t-s) A^*} \CV_t \bar Z_t dt + \sum_{i=1}^n \int_s^T e^{(t-s)A^*} K_i^*(t)\bar Q_{1,t}e_i dt+ \int_s^T e^{(t-s)A^*} \CE_t^j \bar Q_{2,t}^j dt   \\
\aad \quad + \int_s^T e^{(t-s)A^*} (T-t)^{-\vartheta} \bar{\mathscr{W}}_t^n dt  -\sum_{i=1}^\infty \int_s^T e^{(t-s)A^*} \bar Q_{1,t} e_i d\beta_t^i  \\
\aad \quad -\int_s^T e^{(t-s) A^*} \bar Q_{2,t}^j dB_t^j  -\int_{\Xi_s^T} e^{(t-s) A^*} \bar Q_{3,t}(\al) \wdt N(d\al,dt),
\eea
where $\bar{\mathscr{W}}^n$ is defined as $
\bar{\mathscr{W}}_t^n := (T-t)^{\vartheta} \sum_{i=1}^\infty K_i^*(t) \bar Q_{1,t} e_i - (T-t)^{\vartheta} \sum_{i=1}^n K_i^*(t) \bar Q_{1,t}e_i.$

Hence, \lemref{lem:bar-P} implies $(\bar Z,\bar Q_{1},\bar Q_2^j, \bar Q_3\cd)$ is the unique mild solution of \eqref{CO-mild} with $\eta=0, \xi_t=(T-t)^{-\vartheta}\bar{\mathscr{W}}_t^n$. Thus, \eqref{dual-CO} gives
\bea\ad\!\!
\wdh \EE\bqv{\bar Z_s,y}+\wdh \EE\int_s^T \Big[\bqv{\bar Q_{1,t},\Ga_t^\dag}_{\CL_2(H)}+ \bqv{\bar Q_{2,t}^j, \CR_{t}^{2,j}} dt + \int_{\Xi} \bqv{\bar Q_{3,t}(\al), \CR_t^3(\al)} \pi(d\al)\Big] dt \\
\aad\!\! = \wdh \EE\int_s^T \bqv{(T-t)^{-\vartheta}\wdt \CY_t^{n,\Ga^\dag}, \bar{\mathscr{W}}_t^n}_{\CL_2(H)}dt,
\eea
where $\wdt \CY_t^{n,\Ga^\dag}$ is the mild solution of \eqref{CY-m} with $\chi=\CR^\dag=0$. Sicne $\bar{\mathscr{W}}^n$ converges weakly to zero by \lemref{lem:Q1-m}, we take $n\to \infty$ to above equality and have
\bea\!\!\!\!
\wdh \EE\qv{\bar Z_s,y}+\wdh \EE\int_s^T \big[\qv{\bar Q_{1,t},\Ga_t}_{\CL_2(H)}+ \qv{\bar Q_{2,t}^j, \CR_{t}^{2,j}} dt + \int_{\Xi} \qv{\bar Q_{3,t}(\al), \CR_t^3(\al)} \pi(d\al)\big]dt=0.
\eea
This establishes uniqueness and completes the proof.
\end{proof}

\subsection{Trace class regularity of $Q_{1}$}
We now show that the martingale integrand $Q_{1}$ belongs to the trace class, a  property that is
essential for formulating the stochastic maximum principle.
  Note that $Q_1$ naturally takes values in the Hilbert-Schmidt class, the Hilbert-Schmidt smoothing property of the semigroup $e^{tA}$ in \assmref{ass1} enables us to control the singular drift term in the adjoint BSPDE
  through a duality argument, to establish weighted $\CL_1(H)$-bounds for $Q_1$.
\begin{prop}
Let $(Z,Q_{1}, Q_2^j, Q_3\cd)$ be the unique mild solution of \eqref{CZ}. We have
\bea
\wdh \EE\int_0^T (T-t)^{2\vartheta} |Q_{1,t}|_{\CL_1(H)}^2 dt \leq C\big(\wdh \EE|\eta|^2+ \wdh \EE\int_0^T |\CJ_\tau|^2 dt \big).
\eea
\end{prop}

\begin{proof}
Similar to the argument of \cite[Appendix]{FHT18},
we provide the details of the proof in what follows. Since $Q_{1,t}\in \CL_2(H),\wdh\QQ$-a.s. and thus compact, it can be decomposed as $Q_{1,t}=\sum_{\iota=1}^\infty a_\iota(t) e'_\iota(t) \qv{e_\iota(t),\cdot}$, where $a_\iota(t)\in \rr$, and $\{e'_\iota(t)\}_\iota, \{e_\iota(t)\}_{\iota}$ are orthonormal bases of  $H$. The process $a_\iota(t),e'_\iota(t), e_\iota(t)$ are chosen to be progressively measurable. Let
\beqa{Ga-n}
\Ga_t^n= \mathfrak{p}(t)\sum_{\iota=1}^n \text{sgn}(a_\iota(t)) e'_\iota(t)\qv{e_\iota(t),\cdot}
\eeqa
where $\mathfrak{p}$ is an arbitrary positive real-valued bounded progressively measurable process. Since $|\Ga_t^n|_{\CL(H)} \leq \mathfrak{p}(t)$ and $\Ga_t^n$ has rank $n$, the process $\Ga^n$ is also bounded in $\CL_2(H)$.

Let $\CY^n$ denote the solution of \eqref{CY} with $s=y=\chi=\CR^\dag=\CR^2=\CR^3=0$ and $\Ga^\dag=\Ga^n$. By the duality 
%formula 
\eqref{dual-n}, we have
%\bea
$\wdh \EE\int_0^T \qv{Q_{1,t},\Ga_t^n}_{\CL_2(H)}dt = \wdh \EE\qv{\eta,\CY_T^n}+ \wdh \EE\int_s^T \qv{\CJ_t, \CY_t^n} dt.$
%\eea
Computing $\qv{Q_{1,t},\Ga_t^n}$ and applying the  \holder inequality to estimate the right-hand side of above equality using \eqref{CY-L2}, we obtain
\bea
\wdh \EE\int_0^T \sum_{\iota=1}^n |a_\iota(t)| \mathfrak{p}(t)dt
& \leq C \big(\wdh \EE|\eta|^2\big)^{1/2} \big(1+\wdh \EE\int_0^T(T-t)^{-2\vartheta} \mathfrak{p}^2(t)dt \big)^{1/2} \\
&  \quad + C \big(\int_0^T \wdh \EE|\CJ_t|^2 dt \big)^{1/2} \big(1+ \wdh \EE\int_0^T \mathfrak{p}^2(t) dt \big)^{1/2} .
\eea
Letting $n\to\infty$ and noticing that $|Q_{1,t}|_{\CL_1(H)}=\sum_{\iota=1}^\infty |a_\iota(t)|$, we have
\bea
\wdh \EE\int_0^T \mathfrak{p}(t) |Q_{1,t}|_{\CL_1(H)}dt \leq C_{\mathfrak{p},\eta}^{1/2} \Big(1+
\wdh \EE\int_0^T (T-t)^{-2\vartheta} \mathfrak{p}^2(t) dt \Big)^{1/2},
\eea
where $C_{\mathfrak{p},\eta}=C(\EE|\eta|^2 + \EE\int_0^T |\CJ_t|^2 dt)$. Define $\wdt{\mathfrak{p}}(t)=(T-t)^{-\vartheta} \mathfrak{p}(t)$, we have
\bea
\wdh \EE \int_0^T \wdt{\mathfrak{p}}(t) [(T-t)^{-\vartheta}|Q_{1,t}|_{\CL_1(H)}] dt \leq C_{\mathfrak{p},\eta}^{1/2}(1+|\wdt{\mathfrak{p}}|_{L_\CP^2(\Omega\times [0,T])}).
\eea
Since $\mathfrak{p}$ is arbitrary, the proof is complete.
\end{proof}

\section{Malliavin calculus approach}\label{sec:Mal}
In this section, we extend our results to the non-Markovian system allowing all coefficients $F,G_1,G_2^j,\Theta,g,h^j$ in \eqref{sys1} and $L,\phi$ in the cost \eqref{cost} depend on $\omega\in \Omega$. We assume that for any fixed $(x,u)$, the map $\omega\mapsto F(t,x,u,\omega)$ is $\CF_t$-adapted with similar
adaptedness
holding for $G_1,G_2^j,\Theta,g,h^j,L,\phi$. For the derivation of stochastic
maximum principle for optimal control problems using Malliavin calculus, we refer to the work \cite{Ben92,WWX13,MMPB13,OS10,Tan98,BEK89} and references therein.

For any $\kappa\in [1,+\infty)$, we define the Schatten class $\mathbb{S}^\kappa$ as the space of compact operators $\fB: H\to H$ such that
$\|\fB\|_\kappa :=\big[\sum_{i=1}^\infty\la_i(\fB^*\fB)^{\kappa/2}\big]^{1/\kappa}< + \infty$,
where $\la_i(\fB^* \fB)$ is the $i$-th eigenvalue of $\fB^* \fB$. For $\kappa \in [1,\infty)$, the space $\mathbb{S}^\kappa$ is a UMD Banach space; furthermore, it is an $M$-type 2 Banach space whenever $\kappa\geq 2$; see \cite[Lemma 5.1]{GP24}. We refer the reader to \cite{NVW08,Rud04} and references therein for a detailed treatment of stochastic integration on UMD Banach spaces.

\subsection{Preliminaries for Malliavin calculus}
We recall some basic notion and results
of Malliavin calculus
in \cite{Nua06,NOP08}. Let $V$ be a real separable Hilbert space. For a Malliavin differentiable $V$-valued random variables $\zeta$, we denote by
$D_t^W \zeta, D_t^{B^j}\zeta$, and $D_{t,\al}^{\wdt N}\zeta$ its Malliavin derivatives %of a differentiable $V$-valued random variable $\zeta$
with respect to $W$, $B^j$ at $t$, and to the compensated Poisson random measure
$\wdt N(\cdot,\cdot)$ at $(t,\al)$, respectively. Let $\DD_{1,2}(V)$ be the space of all $V$-valued random variables that are Malliavin differentiable with respect to $W, B^j$, and $\wdt N$. And let $\LL_{1,2}(V)$ be the set of all progressively measurable processes such that (i) for a.e. $t\in [0,T]$, $\zeta(t,\cdot)\in \DD_{1,2}(V)$; (ii) $(t,\omega) \mapsto D_s^W \zeta(t,\omega) \in L_\CF^2([0,T];V)$ admits a progressively measurable version such that $\wdh \EE\big(\int_0^T |\zeta(t,\omega)|_V^2 dt + \int_0^T \int_0^T \|D_s^W \zeta(s,\omega)\|_{\CL_2(H,V)}^2 ds dt \big)<+\infty$ with analogous properties holding for Malliavin derivatives with respect to $B^j$ and $\wdt N$. The tool for analysis is the following duality formula: for  $\zeta\in \DD_{1,2}(V)$, it holds that
\beqa{Mal-BW}
& \wdh \EE\big[\bqv{\zeta, \int_0^T \wdh \phi(s)dW_s}_H \big]  = \wdh \EE\int_0^T \bqv{\wdh \phi(s), D_s^W \zeta }_{\CL_2(H)}
ds\\
&
\wdh \EE\big[\bqv{\zeta,\int_0^T \wdt \phi(s)dB^j_s}_V \big] =\wdh \EE\int_0^T \bqv{\wdt \phi(s), D_s^{B^j} \zeta}_V  ds,\\
& \EE\big[\bqv{\zeta, \int_0^T \int_{\Xi} \wdt \psi(s,\al) \wdt N(d\al,ds)}_V \big] = \wdh \EE\int_0^T \int_{\Xi} \bqv{\wdt \psi(s,\al), D_{s,\al}^{\wdt N} \zeta}_V \pi(d\al)ds,
\eeqa
for any $\CF_s$-predictable process $\wdh \phi(s) \in \CL_2(H)$ and $\wdt\phi(s), \wdt\psi(s,\al) \in V$ such that the integrals on the right converges absolutely. Furthermore, we also need the following property of Malliavin derivatives: if $\zeta\in \DD_{1,2}(V)$ is $\CF_s$-measurable, then $D_t^W \zeta=D_t^{B^j}\zeta = D_{t,\al}^{\wdt N} \zeta=0, \forall\, t>s$. In what follows, $V$ is taken to be $\rr$ and $H$, as appropriate.

\subsection{Stochastic flows for linear SPDEs with jumps}\label{sec:flow} To proceed,  we consider the following  linear SPDEs with jumps for $t\in [s,T]$:

\beqa{SPDE-jump}
d x_t^1 & = \big[A x_t^1 + \wdh O_t x_t^1 \big] dt + \nabla_x \wdh G_1(t)x_t^1 dW_t  \\
&\quad +  \nabla_x \wdh G_2^j(t) x_t^1 dB_t^j+ \int_\Xi \nabla_x \wdh \Theta(t,\al) x_t^1 \wdt N(d\al,dt), \quad x_s^1=x \in H.
\eeqa
The solution of \eqref{SPDE-jump} is understood in the mild sense as usual and has a \cadlag trajectories in $H$ denoted by $x_t^{1,x}, t \geq s$. Let $\Dl:=\{(s,t):0 \leq s \leq t <+\infty\}$.

\begin{defn}\label{def:flow}\rm{
We say that \eqref{SPDE-jump} defines a stochastic flow if there exists a mapping $\Phi: \Dl \times \Omega \to \CL(H)$ such that: (i) for every $s\geq 0$ and $x\in H$, the process $\Phi(t,s,\cdot)(x), t\geq s$ has a \cadlag trajectories in $H$, $\wdh \QQ$-a.s; (ii) for every $s \geq 0$ and $x\in H$, we have $\Phi(t,s)(x)=x_t^{1,x}$ for all $t\geq s$, $\wdh \QQ$-a.s.; (iii)  for all $0\leq s\leq t \leq r$ and $\omega\in \Omega$, $\Phi(r,t;\omega)\circ \Phi(t,s;\omega)= \Phi(r,s;\omega)$.
}
\end{defn}

Our goal of this subsection is to represent the solution of \eqref{SPDE-jump} above by \textit{stochastic flows}, namely, $x_t^1=\Phi(t,s)x_s^1$, where $\Phi(t,s)$ denotes a \textit{random evolution operator} satisfying the following \textit{operator-valued} SPDEs with jumps written formally as
\vspace{-0.2cm}
\beqa{eq-Phi}
d\Phi(t,s)& = \big[\CA \Phi(t,s) dt + \wdh O_t \Phi(t,s) \big] dt+ \disp \sum_{i=1}^\infty K_i(t)\Phi(t,s)\, d\beta^i_t  \\
&\; + \nabla_x \wdh G_2^j(t)\Phi(t,s)dB_t^j + \int_\Xi \nabla_x \wdh \Theta(t,\al) \Phi(t,s) \wdt N(d\al,dt),\, \Phi(s,s)=I,
\eeqa
where $\CA$ is the infinitesimal generator of a semigroup $\CS(t)$, defined by
$\CS(t) \Phi:=e^{tA}\circ \Phi$ for $t \geq 0$ and $\Phi\in \CL(H)$ with $\circ$ denoting the composition of operators.

It is natural to seek a stochastic flow $\Phi$ satisfying \eqref{eq-Phi} on the space of bounded operators. However, when $H$ is infinite dimensional, a proper theory of stochastic integration on $\mathcal{L}(H)$ is
not available.
Instead, stochastic integration can be developed on smaller operator spaces, such as the Hilbert-Schmidt class, as developed by Flandoli \cite{FU90,Fla24}, and more generally on the Schatten classes, as considered in \cite{GP24}. We further note that \eqref{eq-Phi} is driven by infinitely many Brownian motions $\{\beta^i\}_{i=1}^\infty$. Consequently, the approach of \cite{MMPB13}, which relies on Kunita's theory of stochastic flows in Euclidean space, is not applicable here, as it requires finitely many driving Brownian motions. Our method is related in spirit to \cite{GP24}, where stochastic integration on UMD Banach spaces plays an central role. Nevertheless, the results of \cite{GP24} cannot be applied directly in our setting, since their analysis assumes time-independent operators $K_i$ satisfying the summability condition
$\sum_{i=1}^\infty|K_i|_{\CL(H)}^2 <+\infty$, which is violated in our framework; see \remref{rem:est-Ki}. Accordingly, additional
adaptions are required.

By \cite[Lemma 5.3]{GP24}, $\CS=\{\CS(t)\}$ forms a semigroup on $\CL(H)$, but it does not, in general,  satisfy the $C_0$-property on $\CL(H)$. Nevertheless, when restricted to the Schatten class $\bS^\kappa$, it becomes a $C_0$-contraction semigroup.

\begin{thm}\label{thm:flows}Assume that there exists constants $\vartheta'<1/2$ and $C>0$ such that $\|e^{tA}\|_{\kappa} \leq C t^{-\vartheta'}, \forall\, t\in [0,T]$. Let Assumptions \ref{ass} hold with space $\CL_2(H)$ and constant $\vartheta$ in (H2) replaced by $\bS^\kappa$ and $\vartheta'$, respectively. Then for any $s \geq 0$, \eqref{eq-Phi} has a unique mild solution in $\mathbb{S}^\kappa$. Moreover $(s,+\infty) \ni t\mapsto \Phi(t,s)\in \mathbb{S}^\kappa$ is \cadlag $\wdh\QQ$-a.s. Hence, $\Phi$ is the stochastic flow corresponding to \eqref{SPDE-jump}.
\end{thm}

\begin{proof}
To simplify the exposition, we set $\wdh O_t = 0$ and $\nabla_x \wdh G_2^j(t) = 0$ in \eqref{SPDE-jump} and \eqref{eq-Phi}. The treatment of the general case is analogous. Fix $0\leq s <T<+\infty$, and let $\Phi(t):=\Phi(t,s)$ for simplicity. Denote by $\mathbb{T}(\bS^\kappa)$ be the set of all adapted measurable process $\Phi:(s,T] \times \Omega \to \bS^\kappa$ such that $\wdh \EE\int_s^T \|\Phi(t)\|_{\kappa}^2 dt <+\infty$. On $\TT(\bS^\kappa)$, consider its equivalent norm $\tnorm{\Phi}_\varrho\!:=\big[\wdh \EE\int_s^T e^{-\varrho t} \|\Phi(t)\|_{\kappa}^2 dt \big]^{1/2}$, for some $\varrho\geq 0$.

By our assumptions, there exists a constant $C>0$ such that for any $\Phi\in \TT(\bS^\kappa)$,

\bea\ad\!\!\!\!\!\!\!
\sum_{i=1}^\infty \wdh \EE\int_s^t \|\CS(t-r)K_i(r)\Phi(r)\|_{\kappa}^2 dr \leq \sum_{i=1}^\infty \wdh \EE\int_s^t \|e^{(t-r)A}\nabla_x \wdh G_1(r)\|_{\CL(H;\bS^\kappa)}^2 |\Phi(r)e_i|_H^2 dr  \\
\aad\!\!\!\!\!\!\! \leq C\, \wdh \EE  \int_s^t (t-r)^{-2\vartheta'} \|\Phi(r)\|_{2}^2 dr \leq C\,  \wdh \EE\int_s^t (t-r)^{-2\vartheta'} \|\Phi(r)\|_{\kappa}^2 dr,
\eea
where the last line follows from Fubini's theorem and $\|\Phi(r)\|_2\leq \|\Phi(r)\|_\kappa,\forall\, \kappa \geq 2$. By the ideal property of Schatten class $\bS^\kappa$ and assumption (H6), we also have
\vspace{-0.2cm}
\bea\ad\!\!\!\!\!\!\!\!\!
\wdh \EE\int_s^t\!\! \int_\Xi \| e^{(t-r)A}\nabla_x \wdh \Theta(r,\al)\Phi(r)\|_{\kappa}^2 \pi(d\al)dr \\
\aad\!\!\!\!\!\!\!\!\! \leq \wdh \EE\!\! \int_s^t\!\! \int_\Xi\! \|e^{(t-r)A}\|_{\kappa}^2 \|\nabla_x \wdh \Theta(r,\al)\|_{\CL(H)}^2 \|\Phi(r)\|_{\kappa}^2 \pi(d\al) dr \leq C\, \wdh \EE\! \int_s^t\!\! (t-r)^{-2\vartheta'}\! \|\Phi(r)\|_{\kappa}^2 dr.
\eea
Therefore, the mapping
\bea
\CT \Phi(t) &\!\!\! :=\CS(t-s)I + \sum_{i=1}^\infty \int_s^t \CS(t-r) K_i(r)\Phi(r)d \beta_r^i   \\
& + \int_s^t\int_\Xi \CS(t-r) \nabla_x \wdh \Theta(r,\al)\Phi(r) \wdt N(d\al,dr)
\eea
is well-defined from $\mathbb{T}(\bS^\kappa)$ to $\mathbb{T}(\bS^\kappa)$. Hence, for $\varrho>0$ large enough, $\CT$ is contraction on $(\TT(\bS^\kappa),|||\cdot|||_\varrho)$. Indeed, for any $\Phi_1,\Phi_2\in \TT(\bS^\kappa)$, we have
\bea\!\!
\tnorm{\CT(\Phi_1)-\CT(\Phi_2)}_\varrho^2  \ad\!\!\! \leq C_T \wdh \EE\int_s^T e^{-\varrho t} \int_s^t (t-r)^{-2\vartheta'} \|\Phi_1(r)-\Phi_2(r)\|_{\kappa}^2 dr dt  \\
\ad \!\!\! \leq C_T\, \wdh \EE \int_s^T e^{-\varrho r} \|\Phi_1(r)-\Phi_2(r)\|_{\kappa}^2 \int_r^T e^{-\varrho(t-r)}(t-r)^{-2\vartheta'} dt\, dr \\
\ad \!\!\! \leq \frac{C_T}{\varrho^{1-2\vartheta'}}\  \check{\Ga} (1-2\vartheta')\ \wdh \EE\int_s^T e^{-\varrho r}\|\Phi_1(r)-\Phi_2(r)\|_{\kappa} dr,
\eea
where we used the fact that $\int_0^\infty e^{-\varrho t} t^{-2\vartheta'}dt = \varrho^{2\vartheta'-1} \check{\Ga}(1-2\vartheta')$ with $\check{\Ga}$ being the Gamma function. By the  Banach fixed point theorem, there exists a $\Phi(s,\cdot)\in \mathbb{T}(\bS^\kappa)$ such that $\CT(\Phi(s,\cdot))=\Phi(s,\cdot)$. For any $\Phi\in \TT(\bS^\kappa)$ and $t\geq s$, the \cadlag property of $\Phi$ follows from the fact that the stochastic integral $\sum_{i=1}^\infty \int_s^t \CS(t-r) K_i(r)\Phi(r)d\beta_r^i$ has  continuous paths in $\bS^\kappa$ and $\int_s^t \int_\Xi \CS(t-r) \nabla_x \wdh \Theta(r,\al) \Phi(r) \wdt N(d\al,dr)$ has \cadlag paths in $\bS^\kappa$ by using standard stochastic factorization method in \cite{DZ14}; see also \cite{MPR10} and similar arguments in \thmref{thm:sup-norm}. The stochastic flow property in \defref{def:flow} follows directly. Thus the proof is complete.
\end{proof}

\subsection{Stochastic maximum principles} In this subsection, we consider $\kappa=2$ and $\vartheta'=\vartheta$ in \thmref{thm:flows} and
let $\Phi(\cdot,s)$ be the solution of \eqref{eq-Phi} in $\bS^2$. We introduce some new adjoint processes. Define $
\aleph(t):=\phi(\wdh x_T)+\int_t^T\int_{\Xi} \wdh L(s,\al) \pi(d\al) ds$ and
\beqa{def-Pi}
\Pi_t:= e^{(T-t)A^*}\big[\nabla_x \phi(\wdh x_T)-\nabla_x^* f(\wdh x_T) \ell_T \big] + \int_t^T \int_{\Xi} e^{(s-t)A^*} \nabla_x \wdh L(s,\al) \pi(d\al) ds.
\eeqa
By the Riesz representation theorem, we identify $\CL(H;\rr)$ with $H$ and hence view $\Pi_t$ as an $H$-valued process. The same convention applies to other quantities taking values in $\CL(H;\rr)$ whenever needed.
Recalling $\wdh O_t$ in \eqref{def-OR}, we introduce
\beqa{def-SH}
\!\! \nabla_x \wdh \SH(s)&\!\!\!\!\!\!\disp :=\int_\Xi \Big\{ \big[\wdh O_s \big]^* \Pi_s + \big[\nabla_x \wdh G_1(s) \big]^* D_s^W \Pi_s + \big[\nabla_x \wdh G_2^j(s)\big]^* D_s^{B^j} \Pi_s \\
& \qquad\!\!\!\!\! + \big[\nabla_x \wdh \Theta(s,\al)\big]^* D_{s,\al}^{\wdt N}\Pi_s + \big[\nabla_x \wdh h^j(s)\big]^* D_s^{B^j} \aleph(s) -\big[\nabla_x \wdh g(s,\al)\big]^* \ell_s \Big\} \pi(d\al)  \\
\Psi(t,s) & \!\!\!\!\!\!: = [\Phi(t,s)]^* \nabla_x \wdh \SH(s) \\
M_s & \!\!\!\!\!\! := \Pi_s + \int_s^T \Psi(t,s)dt,\ N_{1,s}: = D_s^W M_s,\ N_{2,s}: = D_s^{B^j} M_s,\ N_{3,s}: =D_{s,\al}^{\wdt N} M_s.
\eeqa

\begin{thm}
Assume that Assumptions \ref{ass}, \ref{ass1}, and \ref{ass:cost} hold. Suppose \eqref{ellt} admits a unique solution $\ell \in L_\CP^2([0,T]\times \Omega;\DD_{1,2}(\rr^d))$. Moreover, we assume $\phi, L\in \DD_{1,2}(\rr)$, $\nabla_x \phi, \nabla_x L$ and $\Psi(t,s)$ are in $\LL_{1,2}(H)$ for all $0\leq s\leq t \leq T$. Then we have
\bea\ad
\disp \wdh \EE\Big[ \int_\Xi \Big\{ [\wdh R_s]^* M_s+  [\nabla_u \wdh G_1(s)]^* N_{1,s}+ [\nabla_u \wdh G_2^j(s)]^*  N_{2,s} \\
\ad \qquad+ [\nabla_u \wdh \Theta(s,\al)]^* N_{3,s} + [\nabla_u \wdh h^j(s)]^* D_s^{B_j}\aleph(s)- [\nabla_u \wdh g(s,\al)]^* \ell_s \Big\} \pi(d\al) \,\Big|\,\CF_s^Y \Big] =0,
\eea
where $\wdh R_s$ is defined in \eqref{def-OR} and $M_s,N_{1,s},N_{2,s}, N_{3,s}$ are defined in \eqref{def-SH}.
\end{thm}

\begin{proof}
Since $\wdh u$ is a local minimum of $J$, \propref{prop:var-cost} implies $0=[dJ(u^\e)/d\e]_{\e=0}\\ =I(v)$. We rewrite the first line of $I(v)$ as $\wdh \EE\big[(\nabla_x \phi(\wdh x_T) x_T^1 - \ell_T \nabla_x f(\wdh x_T)x_T^1) + \ell_T \nabla_x f(\wdh x_T) x_T^1 +\phi(\wdh x_T) \La_T \big]$. From \eqref{eq-Ga}, we have
\beqa{phi-Ga}
\wdh \EE\big[\phi(\wdh x_T)\, \La_T \big] & = \wdh \EE \big\{\phi(\wdh x_T) \int_0^T\big  [\nabla_x \wdh h^j(t) x_t^1 +\nabla_u \wdh h^j(t) v_t\big]\, dB_t^j \big\} \\
& = \wdh \EE \int_0^T \bqv{\nabla_x \wdh h^j(t) x_t^1+\nabla_u \wdh h^j(t)v_t, D_t^{B^j} \phi(\wdh x_T)}_\rr dt
\eeqa
and Fubini's theorem implies
\beqa{Ga-L}&\!\!\!\!
\wdh \EE \big[\int_{\Xi_T}\La_t\, \wdh L(t,\al)\pi(d\al)dt \big]  \\
&\!\!\!\! =\wdh \EE \int_{\Xi_T} \wdh L(t,\al) \int_0^t \big[\nabla_x \wdh h^j(s) x_s^1 +\nabla_u \wdh h^j(s) v_s \big] dB_s^j \,\pi(d\al)\, dt \\
&\!\!\!\!=\wdh \EE \int_0^T \int_0^t \big[\nabla_x \wdh h^j(s) x_s^1+\nabla_u \wdh h^j(s)v_s \big] \int_\Xi \big(D_s^{B^j} \wdh L(t,\al) \big) \pi(d\al) \, ds\, dt  \\
&\!\!\!\! = \wdh \EE \int_0^T \big(\int_t^T\int_\Xi D_t^{B^j} \wdh L(s,\al) \pi(d\al)\,ds \big) \big[\nabla_x \wdh h^j(t) x_t^1 +\nabla_u \wdh h^j(t) v_t\big]\, dt.
\eeqa
Combining \eqref{phi-Ga} and \eqref{Ga-L} and recalling the definition of $\aleph$ yields
\beqa{phi-Ga-L}&
\wdh \EE\big[\phi(\wdh x_t)\, \La_T+ \int_{\Xi_T} \La_t\, \wdh L(t,\al) \pi(d\al) dt \big] \\
&  =\wdh \EE \int_0^T \bqv{\nabla_x \wdh h^j(t)x_t^1+ \nabla_u \wdh h^j(t)v_t, D_t^{B^j}\phi(\wdh x_T)+\int_t^T \int_{\Xi} D_t^{B^j} \wdh L(s,\al) ds}_{\rr} dt \\
& =\wdh \EE \int_0^T \bqv{\nabla_x \wdh h^j(t) x_t^1 +\nabla_u \wdh h^j(t)v_t, D_t^{B^j} \aleph(t)}_{\rr}dt.
\eeqa

Recalling \eqref{X1-re}, rewriting $x_T^1$ in mild form, and applying \eqref{Mal-BW} give
\beqa{phi-f}\ad
\wdh \EE\big\{ \bqv{\nabla_x \phi(\wdh x_T) - \nabla_x^* f(\wdh x_T)\ell_T, x_T^1} \big\} \\
\aad = \wdh \EE \Big\{\Bqv{\nabla_x \phi(\wdh x_T) - \nabla_x^* f(\wdh x_T) \ell_T,  \int_0^T e^{(T-t)A} \big[\wdh O_t x_t^1 + \wdh R_t v_t \big] dt \\
\aad\;\, + \int_0^T e^{(T-t)A}\Big[\big(\nabla_x \wdh G_1(t) x_t^1 + \nabla_u \wdh G_1(t)v_t\big) dW_t + \big( \nabla_x \wdh G_2^j(t) x_t^1 + \nabla_u \wdh G_2^j(t) v_t \big) dB_t^j\Big]\\

\aad\;\, +\int_{\Xi_T} e^{(T-t)A}\big(\nabla_x \wdh \Theta(t,\al)x_t^1 + \nabla_u \wdh \Theta(t,\al)v_t\big) \wdt N(d\al,dt)}
\Big\} \\
\aad = \wdh \EE \Big\{\int_0^T \bqv{\nabla_x \phi(\wdh x_T) - \nabla_x^* f(\wdh x_T)\ell_T , e^{(T-t)A} \big[ \wdh O_t x_t^1 +\wdh R_t v_t \big]} \\
\aad\qquad\quad + \bqv{e^{(T-t)A} \big( \nabla_x \wdh G_1(t) x_t^1 + \nabla_u \wdh G_1(t) v_t\big), D_t^W \big[\nabla_x \phi(\wdh x_T) - \nabla_x^* f(\wdh x_T)\ell_T \big]}_{\CL_2(H)} \\
\aad\qquad\quad + \bqv{ e^{(T-t)A} \big( \nabla_x \wdh G_2^j(t) x_t^1 + \nabla_u \wdh G_2^j(t) v_t \big), D_t^{B^j} \big[\nabla_x \phi(\wdh x_T) -  \nabla_x^* f(\wdh x_T)\ell_T \big] }_H \\
\aad \qquad\quad + \int_{\Xi} \bqv{e^{(T-t)A}(\nabla_x \wdh \Theta(t,\al)x_t^1 + \nabla_u \wdh \Theta(t,\al)v_t), \\
\aad \qquad\qquad\qquad D_{t,\al}^{\wdt N}[\nabla_x \phi(\wdh x_T)-\nabla_x^* f(\wdh x_T) \ell_T]}_H \pi(d\al) dt\Big\}.
\eeqa
A similar argument
%arguments
by rewriting $x_t^1$ satisfying \eqref{X1-re} in mild form implies
\beqa{L-X}
\wdh \EE \int_{\Xi_T } \nabla_x \wdh L(t,\al) x_t^1 \pi(d\al) dt = \SE_1+\SE_2+\SE_3+\SE_4,
\eeqa
where
\bea
\SE_1 \ad\!\!\!:=\wdh \EE \int_{\Xi_T} \int_0^t \nabla_x \wdh L(t,\al) \big\{e^{(t-s)A}(\wdh O_s x_s^1+ \wdh R_s v_s)\big\} ds\, \pi(d\al)dt \\
\SE_2 \ad\!\!\! := \wdh \EE \int_{\Xi_T}\int_0^t \bqv{ e^{(t-s)A} \big[\nabla_x \wdh G_1(s) x_s^1 + \nabla_u \wdh G_1(s) v_s\big], D_s^W \nabla_x \wdh L(t,\al)}_{\CL_2(H)}  ds\, \pi(d\al)dt  \\
\SE_3 \ad\!\!\! := \wdh \EE \int_{\Xi_T}\int_0^t \bqv{ e^{(t-s)A} \big[\nabla_x \wdh G_2^j(s) x_s^1 + \nabla_u \wdh G_2^j(s) v_s \big], D_s^{B^j} \nabla_x \wdh L(t,\al)}_H ds\, \pi(d\al)dt\\
\SE_4 \ad\!\!\! := \wdh \EE \int_{\Xi_T}\int_0^t \bqv{e^{(t-s)A} \big[\nabla_x \wdh \Theta(t,\al)x_s^1 + \nabla_u \wdh \Theta(s,\al) v_s \big], D_{s,\al}^{\wdt N} \nabla_x L(t,\al)} ds\, \pi(d\al)dt.
\eea
In terms of $\SE_2$, using
\eqref{Mal-BW}, we can rewrite it as
\bea
\!\!\!\!\SE_2\ad\!\!\!\!\!=
\wdh \EE \int_{\Xi_T} \int_0^t \bqv{e^{(t-s)A} \big[\nabla_x \wdh{G}_1(s) x_s^1 +\nabla_u \wdh{G}_1(s)v_s \big], D_s^W \nabla_x \wdh L(t,\al)}_{\CL_2(H)} ds\, \pi(d\al) dt\\
\aad\!\!\!\!\! =\wdh \EE \int_{\Xi_T}\!\!\int_0^t\!\! \bqv{\nabla_x \wdh{G}_1(s) x_s^1 +\nabla_u \wdh{G}_1(s)v_s, e^{(t-s)A^*} D_s^W \nabla_x \wdh L(t,\al)}_{\CL(H),\CL_2(H)} ds\, \pi(d\al) dt\\
\aad\!\!\!\!\!= \wdh \EE \int_{\Xi_T}\!\!\int_s^T\!\!\! \bqv{\nabla_x \wdh{G}_1(s) x_s^1 +\nabla_u \wdh{G}_1(s)v_s, e^{(t-s)A^*} D_s^W \nabla_x \wdh L(t,\al)}_{\CL(H),\CL_2(H)} dt\, ds \pi(d\al) \\
\aad\!\!\!\!\!= \wdh \EE \int_{\Xi_T}\!\!\int_t^T\!\! \bqv{\nabla_x \wdh{G}_1(t) x_t^1 +\nabla_u \wdh{G}_1(t)v_t, e^{(s-t)A^*} D_t^W \nabla_x \wdh L(s,\al)}_{\CL(H),\CL_2(H)} ds\, dt \pi(d\al) \\
\aad\!\!\!\!\! =\wdh \EE\!\! \int_0^T\!\!\bqv{\nabla_x \wdh{G}_1(t) x_t^1 +\nabla_u \wdh{G}_1(t)v_t, D_t^W\!\! \int_{\Xi_t^T}\! e^{(s-t)A^*} \nabla_x \wdh L(s,\al)\pi(d\al) ds}_{\CL(H),\CL_2(H)} dt.
\eea
Here, the pairing $\qv{\CK_1,\CK_2}_{\CL(H),\CL_2(H)}:=\textnormal{Tr}(\CK_2^* \CK_1)$ is well-defined as $\CK_2^* \CK_1\in \CL_1(H)$ for any $\CK_1\in \CL(H),\CK_2\in \CL_2(H)$. Moreover, Fubini's theorem was applied in the third line to interchange the order of temporal integration.  Applying similar argument of $\SE_2$ to $\SE_3,\SE_4$, \eqref{phi-f}, adding \eqref{phi-f} and \eqref{L-X} together, and recalling $\Pi_t$ in \eqref{def-Pi} yield
\beqa{nabla-L-X}
&
\wdh \EE \Big[\big(\nabla_x\phi(\wdh x_T)-\ell_T \nabla_x f(\wdh x_T) \big) x^1_T + \int_0^T \int_{\Xi} \nabla_x \wdh L(t,\al) x_t^1  \pi(d\al) dt \Big]\\
& = \wdh \EE\! \int_0^T \Big\{\bqv{\wdh O_t x_t^1+ \wdh R_t v_t, \Pi_t}_H+ \bqv{\nabla_x \wdh G_1(t) x_t^1 + \nabla_u \wdh G_1(t) v_t, D_t^W \Pi_t}_{\CL(H),\CL_2(H)} \\
& \qquad\quad + \bqv{\nabla_x \wdh G_2^j(t) x_t^1 + \nabla_u \wdh G_2^j(t) v_t, D_t^{B^j} \Pi_t}_H \\
& \qquad\quad + \int_{\Xi} \bqv{\nabla_x \wdh \Theta(t,\al) x_t^1+ \nabla_u \wdh \Theta(t,\al)v_t, D_{t,\al}^{\wdt N} \Pi_t }_H \pi(d\al) \Big\} dt.
\eeqa

Recall \eqref{dual-y}, inserting \eqref{phi-Ga-L}, \eqref{nabla-L-X}, and \eqref{dual-y} into $I(v)$ in \propref{prop:var-cost} gives
\beqa{Iv-re}
I(v)\ad \!\!\!= \wdh \EE \int_{\Xi_T} \Big\{ \bqv{[\wdh O_t]^* \Pi_t+ [\nabla_x \wdh G_1(t)]^* D_t^W \Pi_t + [\nabla_x \wdh G_2^j(t)]^* D_t^{B^j} \Pi_t  \\
\aad  \!\!\!\quad\; +[\nabla_x \wdh \Theta(t,\al)]^* D_{t,\al}^{\wdt N} \Pi_t + [\nabla_x \wdh h^j(t)]^* D_t^{B^j} \aleph(t)- [\nabla_x \wdh g(t,\al)]^*\ell_t, x_t^1}_H \Big\} \pi(d\al)dt \\
\aad\!\!\! + \wdh\EE \int_{\Xi_T} \Big\{ \bqv{[\wdh R_t]^* \Pi_t + [\nabla_u \wdh G_1(t)]^* D_t^W \Pi_t + +[\nabla_u \wdh G_2^j(t)]^* D_t^{B_j} \Pi_t  \\
\aad\!\!\! \quad\; + [\nabla_u \wdh \Theta(t,\al)]^* D_{t,\al}^{\wdt N} \Pi_t +[\nabla_u \wdh h^j(t)]^* D_t^{B^j} \aleph(t)-[\nabla_u \wdh g(t,\al)]^* \ell_t, v_t}_\CU  \Big\} \pi(d\al) dt.
\eeqa

Following \cite{WWX13,OS10}, we take a particular control $v_t = \nu \indi_{(s,s+\theta]}(t)$ for $\nu=\nu(\omega)$ being a bounded $\CF_s^Y$-measurable random variable and $0\leq s\leq s+\theta \leq T$. In this case, Eq. \eqref{X1} gives $x_t^1=0$ for any $0\leq t\leq s$.
Therefore, \eqref{Iv-re} becomes $I_1(\theta)+I_2(\theta)=0$, where $I_1(\theta)$ and $I_2(\theta)$ are defined by
\bea\!\!
I_1(\theta)& \!\!\! :=\wdh \EE \int_s^T\int_\Xi \bqv{[\wdh O_t]^* \Pi_t + [\nabla_x \wdh G_1(t)]^* D_t^W \Pi_t  + [\nabla_x \wdh G_2^j(t)]^* D_t^{B^j} \Pi_t
\\
&\!\!\!\quad + [\nabla_x \Theta(t,\al)]^* D_{t,\al}^{\wdt N} \Pi_t + [\nabla_x \wdh h^j(t)]^* D_t^{B_j} \aleph(t)-[\nabla_x \wdh g(t,\al)]^* \ell_t, x_t^1}_H \pi(d\al) dt, \\
I_2(\theta)& \!\!\! := \wdh\EE\int_s^{s+\theta}\int_\Xi \bqv{[\wdh R_t]^* \Pi_t + [\nabla_u \wdh G_1(t)]^* D_t^W \Pi_t+[\nabla_u \wdh G_2^j(t)]^* D_t^{B_j} \Pi_t \\
&\!\!\!\quad + [\nabla_u \Theta(t,\al)]^* D_{t,\al}^{\wdt N} \Pi_t  +[\nabla_u \wdh h^j(t)]^* D_t^{B^j} \aleph(t)- [\nabla_u \wdh g(t,\al)]^* \ell_t, \nu}_\CU \pi(d\al) dt.
\eea

For this particular control $v_t =\nu \indi_{(s,s+\theta]}(t)$, we note that for any $t \geq s+\theta$, the first
variational
equation $x^1$ becomes
\bea\ad\!\!\!\!\!\!\!\!
d x_t^1 =\! \{Ax_t^1 + \wdh O_t x_t^1\}dt +\! \nabla_x \wdh G_1(t) x_t^1 dW_t +\! \nabla_x \wdh G_2^j(t) x_t^1 dB_t^j+\!\!\int_\Xi\!\! \nabla_x\wdh \Theta(t,\al) x_t^1 \wdt N(d\al,dt),
\eea
which is a linear SPDE with jumps. Note that $\nabla_x \wdh G_1(t)x_t^1dW_t = \sum_{i=1}^\infty K_i(t)x_t^1d\beta_t^i$.
Therefore, \thmref{thm:flows} implies that $x_t^1=\Phi(t,s+\theta)x_{s+\theta}^1$, where $\Phi(\cdot,s+\theta)$ is the solution of \eqref{eq-Phi} for  initial time $s+\theta$. Noticing again $x_t^1=0$ for all $t\leq s$, the mild form of $x_{s+\theta}^1$ then has the following representation:
\beqa{x-srho}\!\!\!\!\!
x_{s+\theta}^1 \ad\!\!\!= \int_s^{s+\theta} e^{(s+\theta-\tau)A}\Big\{\wdh O_\tau x_\tau^1 d\tau+ \nabla_x \wdh G_1(\tau) x_\tau^1 dW_\tau +  \nabla_x \wdh G_2^j(t) x_\tau^1 dB_\tau^j  \\
\aad \!\!\! \qquad \qquad\qquad\qquad\; +\int_{\Xi} \nabla_x \wdh \Theta(\tau,\al)x_\tau^1 \wdt N(d\al,d\tau)\Big\} \\
\aad+ \int_s^{s+\theta} e^{(s+\theta-\tau)A} \Big\{\wdh R_\tau \nu d\tau + \nabla_u \wdh G_1(\tau)\nu dW_\tau + \nabla_u \wdh G_2^j(\tau) \nu dB_\tau^j \\
\aad \qquad \qquad\qquad\qquad\; +\int_{\Xi}\nabla_u \wdh \Theta(\tau,\al)\nu \wdt N(d\al,d\tau) \Big\} =: x_{s+\theta,1}^1+x_{s+\theta,2}^1.
\eeqa

Recalling the definition of $\nabla_x \wdh \SH(t)$ in \eqref{def-SH} and noting that $x_s^1=0$ lead to
\bea\ad\!\!\!
\frac{d I_1(\theta)}{d\theta} \Big|_{\theta=0} =\frac{d}{d\theta} \wdh\EE \bigg[\int_{s+\theta}^T \bqv{ \nabla_x \wdh\SH(t),  \Phi(t,s+\theta)x_{s+\theta}^1}_H dt \bigg]_{\theta=0} \\
\aad\!\!\!\!\!\!\! =\int_s^T \frac{d}{d\theta} \wdh \EE \Big[\bqv{\nabla_x \wdh \SH(t), \Phi(t,s+\theta) x_{s+\theta}^1 }_H \Big]_{\theta=0}dt = \int_s^T \frac{d}{d\theta} \wdh \EE \Big[\bqv{\Psi(t,s), x_{s+\theta}^1 }_H \Big]_{\theta=0}dt.
\eea
We now express $x_{s+\theta}^1$ in terms of \eqref{x-srho} and apply \eqref{Mal-BW}. It follows that
\bea\ad\!\!\!\!\!
\int_s^T \frac{d}{d\theta} \wdh \EE \Big[\bqv{\Psi(t,s), x_{s+\theta,1}^1} \Big]_{\theta=0}dt  \\
\aad\!\!\!\!\! = \int_s^T \frac{d}{d\theta}\, \wdh \EE \bigg\{ \int_s^{s+\theta}  \Big[\bqv{ \Psi(t,s), e^{(s+\theta-\tau)A}(\wdh O_\tau x_\tau^1)}  \\
\aad\!\!\!\!\! \qquad\quad + \bqv{e^{(s+\theta-\tau)A} \nabla_x \wdh G_1(\tau)x_\tau^1,  D_\tau^W \Psi(t,s)} + \bqv{e^{(s+\theta-\tau)A}\nabla_x \wdh G_2^j(\tau) x_\tau^1,  D_\tau^{B^j} \Psi(t,s)}\\
\aad\!\!\!\!\! \qquad\quad + \int_\Xi \bqv{e^{(s+\theta-\tau)A} \nabla_x \wdh \Theta(\tau,\al) x_\tau^1,  D_{\tau,\al} \Psi(t,s)}\pi(d\al) \Big] d\tau \bigg\}_{\theta=0} dt \\
\aad\!\!\!\!\! = \int_s^T \wdh \EE \Big[ \bqv{ \nabla_x \wdh G_1(s+\theta) x_{s+\theta}^1, D_{s+\theta}^W \Psi(t,s)} + \bqv{ \nabla_x \wdh G_2^j(s+\theta) x_{s+\theta}^1,  D_{s+\theta}^{B_j} \Psi(t,s)} \\
\aad\!\!\!\!\!\quad+\int_\Xi \bqv{\nabla_x \wdh \Theta(s+\theta,\al) x_{s+\theta}^1,  D_{s+\theta,\al}^{\wdt N} \Psi(t,s)} \pi(d\al) + \bqv{\Psi(t,s), \wdh O_{s+\theta} x_{s+\theta}^1} \Big]_{\theta=0} dt =0,
\eea
where the last line used the fact that $x_s^1=0$ after evaluating $\theta=0$. Moreover, a similar argument
as in the above implies
\beqa{I12}\ad\!\!\!
\int_s^T \frac{d}{d\theta} \wdh \EE \Big[\bqv{\Psi(t,s),  x_{s+\theta,2}^1} \Big]_{\theta=0} dt \\
\aad\!\!\! = \wdh \EE\int_s^T \Big[ \bqv{\Psi(t,s), \wdh R_{s+\theta}\nu} + \bqv{\nabla_u \wdh G_1(s+\theta)\nu, D_{s+\theta}^W \Psi(t,s)} \\
\aad\!\!\!\!\!\!\! \quad + \bqv{\nabla_u \wdh G_2^j(s+\theta)\nu,  D_{s+\theta}^{B_j} \Psi(t,s)} +\int_\Xi \bqv{\nabla_u \wdh \Theta(s+\theta,\al)\nu,  D_{s+\theta,\al}^{\wdt N} \Psi(t,s)} \pi(d\al) \Big]_{\theta=0} dt \\
\aad\!\!\! = \wdh \EE \int_s^T \bqv{\Psi(t,s), \wdh R_s\nu} + \bqv{\nabla_u \wdh G_1(s)\nu,  D_s^W \Psi(t,s)} + \bqv{\nabla_u \wdh G_2^j(s)\nu, D_s^{B_j} \Psi(t,s)} \\
\aad \qquad + \int_\Xi \bqv{ \nabla_u \wdh \Theta(s,\al)\nu,  D_{s,\al}^{\wdt N} \Psi(t,s)} \  \pi(d\al) dt.
\eeqa

In light of $I_2(\theta)$, we obtain
\beqa{I2}\disp\!\!
\frac{d I_2(\theta)}{d\theta}\Big|_{\theta=0}  \ad \!\!\!\!\!\! = \wdh \EE \Bqv{\int_\Xi \Big\{ [\wdh R_s]^* \Pi_s + [\nabla_u \wdh G_1(s)]^* D_s^W \Pi_s+[\nabla_u \wdh G_2^j(s)]^* D_s^{B_j} \Pi_s \\
\aad \!\!\!\!\!\! + [\nabla_u \wdh \Theta(s,\al)]^* D_{s,\al}^{\wdt N} \Pi_s\!+\! [\nabla_u \wdh h^j(s)]^* D_s^{B^j} \aleph(s)\!-\![\nabla_u \wdh g(s,\al)]^* \ell_s \Big\} \pi(d\al), \nu}_\CU.
\eeqa
Recall the definition of $M,N_1,N_2,N_3$ in \eqref{def-SH}, combining \eqref{I12} and \eqref{I2} yields
\bea\disp\!\!\!
\frac{dJ(u^\e)}{d\e}\Big|_{\e=0} \ad \!\!\!\! =\wdh \EE \Bqv{\int_\Xi \Big\{ [\wdh R_s]^* M_s+  [\nabla_u \wdh G_1(s)]^* N_{1,s}+ [\nabla_u \wdh G_2^j(s)]^*  N_{2,s} \\
\aad \!\!  + [\nabla_u \wdh \Theta(s,\al)]^* N_{3,s} + [\nabla_u \wdh h^j(s)]^* D_s^{B_j}\aleph(s)-[\nabla_u \wdh g(s,\al)]^* \ell_s \Big\} \pi(d\al), \nu}_{\CU}.
\eea
Since the above equality holds for any bounded $\CF_s^Y$-measurable $\nu$, we conclude that
\bea\ad
\wdh \EE\Big[\int_\Xi \Big\{ [\wdh R_s]^* M_s+  [\nabla_u \wdh G_1(s)]^* N_{1,s}+ [\nabla_u \wdh G_2^j(s)]^*  N_{2,s} \\
\aad \qquad +[\nabla_u \wdh \Theta(s,\al)]^* N_{3,s} + [\nabla_u \wdh h^j(s)]^* D_s^{B_j}\aleph(s)- [\nabla_u \wdh g(s,\al)]^* \ell_s \Big\} \pi(d\al) \,\Big|\,\CF_s^Y \Big] =0.
\eea
Consequently, the proof is complete.
\end{proof}

% \section*{Acknowledgments}
% We would like to two anonymous referees and the editor for their insightful comments and suggestions.

% \bibliographystyle{/Users/hqian/siamplain-abbr.bst}
% \bibliography{/Users/hqian/Library/CloudStorage/OneDrive-WayneStateUniversity/references.bib}

\end{document}